\documentclass[a4paper,final]{siamonline171218}
\pdfoutput=1

\usepackage[utf8]{inputenc}

\usepackage{microtype}
\usepackage[shortlabels]{enumitem}
\setlist[enumerate]{leftmargin=\leftmargini}

\usepackage{amsmath,amsfonts,amssymb}
\usepackage{mathtools}

\usepackage{booktabs}

\usepackage[]{algpseudocode}

\DeclareMathOperator{\pic}{Pic}

\DeclareMathOperator{\pr}{pr}
\DeclareMathOperator{\rk}{rk}

\DeclareMathOperator{\dist}{dist}

\newcommand{\toi}{\hookrightarrow}
\newcommand{\isoto}{\overset{\sim}{\to}}

\newcommand{\Cc}{\mathbb{C}}
\newcommand{\Ff}{\mathbb{F}}

\newcommand{\Ii}{\mathbb{I}}

\newcommand{\Qq}{\mathbb{Q}}
\newcommand{\Pp}{\mathbb{P}}
\newcommand{\Rr}{\mathbb{R}}

\newcommand{\Zz}{\mathbb{Z}}

\newcommand{\cp}{\mathcal{P}}
\newcommand{\cR}{\mathcal{R}}
\newcommand{\cq}{\mathcal{Q}}
\newcommand{\cb}{\mathcal{B}}
\newcommand{\co}{\mathcal{O}}

\newcommand{\dR}{\mathrm{dR}}
\newcommand{\red}{\mathrm{red}}
\newcommand{\hodge}{\mathrm{Hdg}}
\newcommand{\alg}{\mathrm{Alg}}

\renewcommand{\H}{\mathrm{H}}

\newcommand{\p}{\Pp^1}

\newcommand{\ppp}{\Pp^3}

\newcommand{\st}{\ \middle|\ }
\newcommand{\norm}[1]{\lVert #1 \rVert}

\newcommand\eqdef{\smash{\overset{\underset{\mathrm{def}}{}}{=}}}

\makeatletter
\def\imod#1{\allowbreak\mkern10mu({\operator@font mod}\,\,#1)}
\makeatother

\newsiamremark{example}{Example}
\newsiamremark{remark}{Remark}

\def\ud{\,\textrm{d}}
\def\epsilon{\varepsilon}
\def\phi{\varphi}

\title{A numerical transcendental method in algebraic geometry}
\author{Pierre Lairez\thanks{Inria, France} \and Emre Can Sert\"oz\thanks{Max
    Planck Institute MiS,
    Leipzig, Germany}}
\date{\today}

\begin{document}

\maketitle

\begin{abstract}
  Based on high precision computation of periods and lattice reduction techniques, we compute the Picard group of smooth surfaces in~$\mathbb{P}^3$. We also study the lattice reduction technique that is employed in order to quantify the possibility of numerical error in terms of an intrinsic measure of complexity of each surface. The method applies more generally to the computation of the lattice generated by Hodge cycles of middle dimension on smooth projective hypersurfaces. We demonstrate the method by a systematic study of thousands of quartic surfaces (K3s) defined by sparse polynomials. As an application, we count the number of rational curves of a given degree lying on each surface. For quartic surfaces we also compute the endomorphism ring of their transcendental lattice.
\end{abstract}

\section{Introduction}

In ``A transcendental method in algebraic geometry'' \cite{Griffiths_1971},
Griffiths emphasized the role of certain multivariate integrals, known as \emph{periods},
``to construct a continuous invariant of arbitrary smooth projective varieties''.
Periods often determine the projective variety completely and therefore its algebraic invariants. Translating periods into discrete algebraic invariants is a difficult problem, exemplified by the long standing Hodge conjecture which describes how periods determine the algebraic cycles within a projective variety.

Recent progress in computer algebra makes it possible to compute periods with high precision and put transcendental methods into practice. We focus mainly on algebraic surfaces and give a numerical method to compute Picard groups. As an application, we count smooth rational curves on quartic surfaces using the Picard group. These methods apply more generally to hypersurfaces in a projective space of arbitrary dimension and we show some examples.

\subsubsection*{Structure of the Picard group and main results}
There are many curves in a smooth surface $X \subset \ppp_\Cc$, the basic ones are those obtained by intersecting $X$ with another surface $S$ in $\ppp_\Cc$. If $S_1$ and $S_2$ are two surfaces of the same degree then the curve
$C_1 = S_1 \cap X$ can be deformed into the curve $C_2 = S_2 \cap X$ by varying continuously
the coefficients of the defining equation of $S_1$. The curves $C_1$ and $C_2$
are said to be \emph{linearly equivalent}. The notion of linear
equivalence extends to formal $\Zz$-linear combinations of curves and the Picard
group of $X$ is defined by
\[
  \pic(X) \eqdef \Zz \langle \text{algebraic curves in $X$} \rangle / \langle \text{linear equivalence relations} \rangle.
\]
The Picard group is an algebraic invariant that reflects the nature of the algebraic curves lying on~$X$.
It is a free abelian group, that is,
$\pic(X) \simeq \Zz^{\rho}$ for a positive integer $\rho$ called the
\emph{Picard number} of $X$.
As Zariski wrote, ``{The evaluation of $\rho$ for a given surface
  presents in general grave difficulties}''
\cite[p.~110]{zariski--algebraic_surfaces}. 

There is more to the Picard group than the Picard number. The intersection product, which for any two curves $C_1$ and $C_2$ in $X$ associates an integer $C_1\cdot C_2$, induces a bilinear map $\pic(X) \times \pic(X) \to \Zz$.
The intersection product is an intrinsic algebraic invariant of $X$ that is finer than the Picard number. 
There is also an extrinsic invariant in the Picard group, called the \emph{polarization}, recording much of the geometry of $X$ within $\ppp_\Cc$. The polarization is the linear equivalence class of any curve obtained by intersecting $X$ with a plane in $\ppp_\Cc$. The problem we address is then the following:

  \emph{Given the defining equation of $X$, compute the Picard number~$\rho$ of~$X$,
  the $\rho\times\rho$ matrix of the intersection product and the $\rho$
  coordinates of the polarization in some basis of~$\pic(X) \simeq \Zz^\rho$.  }

We approach the problem using \emph{transcendental methods}, that is, we use the complex geometry of the hypersurfaces and compute multivariate integrals on topological cycles, namely the \emph{periods}. 
For surfaces, Lefschetz $(1,1)$ theorem identifies the Picard group of a surface with the
lattice of integer linear relations between periods.
The rank, intersection product and polarization of the Picard group can be computed from
a high precision computation of the periods \cite{sertoz18}
and well-established techniques in lattice reduction. We apply these techniques also 
to the computation of the endomorphism ring of the transcendental lattice in order to compute Charles' gap~\cite{Charles_2014}, see below.
Counting rational curves of a given degree lying on a surface is an interesting
application of the computation of the Picard group with its intersection product and polarization.

The method extends to higher dimensional hypersurfaces in order to compute the group of Hodge cycles.
For a hypersurface $X$ in a projective space of odd dimension $\Pp^{2k+1}_\Cc$ with $k > 1$ there are two interesting objects to study, replacing the Picard group for surfaces under present discussion:  the group of algebraic cycles $\alg^k(X)$ generated by the cohomology classes of $k$-dimensional algebraic subvarieties of $X$ or the group of Hodge cycles $\hodge^k(X)$ generated by integral linear relations between periods. The Hodge conjecture states in greater generality that, after tensoring with rational numbers, the two groups $\alg^k(X)\otimes_\Zz \Qq$ and $\hodge^k(X)\otimes_\Zz \Qq$ coincide~\cite{deligne--millenium}. The resolution of this conjecture is one of the seven Millennium Prize Problems posed by the Clay Institute. Let us point out that this problem is far from being resolved even for hypersurfaces beyond the case of surfaces in $\ppp_\Cc$. Perhaps the current method will allow for experimentation in this direction with the ability to compute $\hodge^k(X)$ together with its intersection product and polarization, see \S\ref{sec:higher_degree_dimension}.

\subsubsection*{Related work}

For surfaces of degree at most three the Picard group does not depend on the defining equations and
the main arguments to compute it have been known since the 19th century~\cite{dolgachev--luigi}.
Starting with surfaces of degree four, the Picard group is sensitive to the defining equations and poses an entirely different kind of challenge,
where a complete solution must be algorithmic in nature.
Noether and Lefschetz~\cite{lefschetz--analysis_situs} proved that a very general quartic surface is expected to have Picard number one.
However, the first quartic with $\rho=1$ defined by a polynomial with integer coefficients appeared in 2007 with van Luijk's seminal paper~\cite{vanLuijk_2007} where he used techniques involving reduction to finite characteristic.
Since then the reduction techniques have been going through a phase of rapid development. The original argument of van Luijk was refined by Elsenhans and Jahnel \cite{ElsenhansJahnel_2011,ElsenhansJahnel_2011a} but the computational bottle neck persisted: in working with surfaces over finite fields the computation of their Zeta function initially required the expensive process of point counting. This bottleneck has been alleviated using ideas from $p$-adic cohomology and gave rise to two different approaches: one dealing directly with the surface \cite{kedlaya04,AbbottKedlayaRoe_2010,CostaHarveyKedlaya_2018} and another which deforms the given surface to a simpler one \cite{lauder04,pacratz-tuitman}. 

To complement the upper bounds coming from prime reductions, lower bounds on the Picard number can be obtained, at least in theory, by enumerating all the algebraic curves in~$X$. One could compute infinite sequences of lower and upper bounds that may eventually determine the Picard number. However, Charles~\cite{Charles_2014} proved that the upper bounds obtained from finite characteristic could significantly overestimate the Picard number and he expressed the gap in terms of the endomorphism ring of the transcendental lattice, see \S\ref{sec:transcendental}. In practice, computing this gap appears to be just as difficult as the computation of the Picard number. However, Charles demonstrated at last that the Picard number of a K3 surface (e.g. a quartic surface) defined with coefficients in a number field is computable. On the theoretical side, effective algorithms have been developed with a broader reach but with low practicability~\cite{Charles_2014,HassettKreschTschinkel_2013,PoonenTestavanLuijk_2015}. There is recent work addressing the issue of practicability~\cite{Festi_2018}.

Concerning numerical methods, high precision computation of periods has been successfully applied to many problems concerning algebraic curves e.g.~\cite{vanWamelen_1999,nils-18}, even with the possibility of \emph{a posteriori} symbolic certification~\cite{CostaMascotSijslingEtAl_2018}.

\subsubsection*{Tools} For the purpose of exposition, we focus mainly on quartic surfaces.
Our techniques, as well as our code, work for hypersurfaces of any degree and dimension, given sufficient computational resources.
The main computational tool on the one hand is an algorithm to approximate periods~\cite{sertoz18},
based on Picard--Fuchs differential equations~\cite{Picard_1902b} and
algorithms to compute them \cite[e.g.]{Chyzak_2000,Koutschan_2010a,Pancratz_2010,Movasati_2007a,Lairez_2016,BostanLairezSalvy_2013} via numerical analytic continuation \cite{Hoeven_2007,Hoeven_2001,ChudnovskyChudnovsky_1990,Mezzarobba_2010,Mezzarobba_2016}.
On the other hand, we use algorithms to compute integer linear relations between
vectors of real numbers \cite{FergusonForcade_1979,LenstraLenstraLovasz_1982,HastadJustLagariasEtAl_1989,BuchmannPohst_1989,FergusonBaileyArno_1999,ChenStehleVillard_2013}.

We used the computer algebra system {\tt Sagemath}~\cite{sagemath} with {\tt ore\_algebra-analytic}\footnote{\url{http://marc.mezzarobba.net/code/ore_algebra-analytic}} for performing numerical analytic continuation~\cite{Mezzarobba_2016} and {\tt Magma}~\cite{BosmaCannonPlayoust_1997},
with {\tt period-suite}\footnote{\url{https://github.com/period-suite/period-suite}}~\cite{sertoz18} and {\tt periods}\footnote{\url{https://github.com/lairez/periods}}~\cite{Lairez_2016}.

\subsubsection*{The reliability of numerical computations}
Although the periods vary continuously with the coefficients of the defining
equation of a surface, the Picard number is nowhere continuous for surfaces of degree at least~$4$ \cite{NL_locus_is_dense}, behaving like the indicator function
of the rational numbers on the real line.
This fact suggests that deducing the Picard group from approximate periods must be hopeless, since the numerical computation of the Picard group is based on finding integer relations between real numbers that we know only with finite precision. 
However, working with sparse polynomials with small rational coefficients, we observe a remarkable tolerance for error, see \S\ref{sec:an_example} for a typical example. Our computations agree with the literature whenever a check is possible. In particular, we compared our Picard number computations against {\tt controlledreduction}\footnote{\url{https://github.com/edgarcosta/controlledreduction}}~\cite{costa-phd} and Shioda's algorithm for Delsarte surfaces~\cite{shioda--delsarte}. 

The possibility of error and its nature is quantified precisely in~\S\S \ref{sec:numer-reconstr-picar}--\ref{sec:intrinsic_error} and compared to an intrinsic measure of complexity of the surface. This measure of complexity is out of reach but its determination would allow certification of the numerical computation.

\subsubsection*{Outline} 

In Section~\ref{sec:periods_and_picard}, we describe the computation of the Picard group of surfaces and the endomorphism ring of the transcendental lattice of quartic surfaces from an approximation of periods.
In Section~\ref{sec:P1inK3} we apply our computations to count smooth rational curves in quartic
surfaces. %
In Section~\ref{sec:numer-reconstr-picar} we describe and analyze a standard procedure to recover integer relations between approximate real vectors.
In Section~\ref{sec:intrinsic_error} we quantify the nature of error in a way that is independent of the methods employed and express it in terms of an intrinsic measure of complexity of the given surface.
Section~\ref{sec:higher_degree_dimension} explains the situation for higher dimensional hypersurfaces where the general idea of the method as explained in Section~\ref{sec:periods_and_picard} applies verbatim. Section~\ref{sec:experiment} summarizes the experimental results obtained for thousands of quartic surfaces.
Section~\ref{sec:completing_homology} explains how we compute the polarization by fleshing out the argument given in~\cite{sertoz18}.

\subsubsection*{Acknowledgments} 

We would like to thank Bernd Sturmfels for putting us together, for his guidance
and constant support. We also would like to thank Alex Degtyarev for explaining
to us how to count rational curves in a K3 lattice. We are also grateful for
insightful conversations with Simon Brandhorst, John Cannon, Edgar Costa,
Stephan Elsenhans, Jon Hauenstein, Marc Mezzarobba, Mateusz Micha{\l}ek, Matthias Sch\"utt and Don Zagier.

\section{Periods and Picard group}\label{sec:periods_and_picard}

\subsection{Principles}
\label{sec:principles}

Following Picard, Lefschetz and Hodge, algebraic curves on a smooth complex surface~$X$ can be
characterized among all topological $2$-dimensional cycles of~$X$ in terms of
multivariate integrals (for a historical perspective, see
\cite{movasati--hodge} and~\cite{Houzel_2002}).

Let~$X\subset \ppp$ be a smooth complex surface.
An algebraic curve $C \subset X$ is supported on a topological $2$-dimensional
cycle. Lefschetz proved that two algebraic curves are topologically homologous
if and only if they are linearly equivalent
\cite{lefschetz--analysis_situs}, see also \cite[Chap.~9]{movasati--hodge}.
In other words, the Picard group comes with a natural inclusion into the homology group
\begin{equation}\label{eq:pic_in_H}
  \pic(X) \hookrightarrow \H_2(X,\Zz).
\end{equation}
This homology group is a topological invariant that depends only on the degree
of~$X$, while $\pic(X)$ is a much finer invariant of $X$.

Recall that for a~$2$-dimension cycle~$\gamma$ and any holomorphic differential $2$-form~$\omega$ on~$X$, the integral~$\int_\gamma \omega$ is well defined on the homology class of $\gamma$.

\begin{theorem}[Lefschetz (1,1) theorem]\label{thm:lefschetz11}
  A homology class $\gamma \in \H_2(X,\Zz)$ is in $\pic(X)$ if and only if $\int_\gamma \omega = 0$ for every holomorphic $2$-form~$\omega$ on~$X$.
\end{theorem}

When~$X$ has degree~$4$, $\H_2(X,\Zz) \simeq \Zz^{22}$ and~$X$
admits a unique non-zero holomorphic $2$-form up to
scaling, which we denote~$\omega_X$.
Given a basis~$\gamma_1,\dotsc,\gamma_{22}$ of~$\H_2(X,\Zz)$, Theorem~\ref{thm:lefschetz11} rewords as
\begin{equation}\label{eq:1}
 \pic(X) = \left\{ (a_1,\dotsc,a_{22}) \in \Zz^{22} \st \sum_{i=1}^{22} a_i
    \int_{\gamma_i} \omega_X = 0 \right\}.
\end{equation}
The integrals~$\int_{\gamma_i} \omega_X$ appearing here are called the \emph{periods} of~$X$.
All periods can be expressed as a sum of integrals in the affine chart $\Cc^3=\{w=1\} \subset \ppp$. For a 2-cycle $\gamma \subset X \cap \{w=1\}$ we can form a thin tube $\tau \subset \Cc^3 \setminus X$ around $\gamma$ so that 
\begin{equation} \int_{\gamma} \omega_X = \frac{1}{2\pi \sqrt{-1}}\int_{\tau} \frac{\ud x\ud y\ud z}{f(x,y,z,1)}, 
\end{equation}
where~$f$ is the degree~$4$ homogeneous polynomial defining~$X$ \cite{griffiths--periods}.

In general, when~$X$ has degree~$d \geq 4$,
$\H_2(X,\Zz)$ has rank~$m=d^3-4d^2+6d-2$ and the space $V$ of holomorphic $2$-forms on~$X$ is of dimension $r=\binom{d - 1}{3}$. Fixing bases $\H_2(X,\Zz)=\Zz\langle \gamma_1,\dots,\gamma_m \rangle$, $V=\Cc\langle \omega_1,\dots,\omega_r \rangle$ and applying Lefschetz~(1,1) theorem we get: 
\begin{equation}\label{eq:2}
  \pic(X) = \left\{ (a_1,\dotsc,a_{m}) \in \Zz^{m} \st \forall 1\leq j \leq r,\, \sum_{i=1}^{m} a_i
    \int_{\gamma_i} \omega_j = 0 \right\}.
\end{equation}

In view of~\eqref{eq:1} and~\eqref{eq:2} we can determine $\pic(X)$ by computing the matrix of periods $[\int_{\gamma_j}\omega_i]_{i,j}$ and then finding integer linear relations between the rows. The algorithm presented in \cite{sertoz18} to compute the periods of $X$ takes care of the first step. We may then use lattice reduction algorithms~\cite[p.~525]{LenstraLenstraLovasz_1982} to find generators for $\pic(X)$.

We briefly recall how periods of $X$ are computed in \cite{sertoz18}. The surface $X$ is put into a single parameter family of surfaces containing the Fermat surface $Y = \{x^d + y^d + z^d + w^d = 0\}$. The matrix of periods along the family vary holomorphically in terms of the parameter and these entries satisfy ordinary differential equations which are computed exactly. The value of this one parameter period matrix---as well as its derivatives---at $Y$ are given by closed formulas involving Gamma functions. The differential equations together with the periods on $Y$ expresses the periods of $X$ as the solution to an initial value problem. This initial value problem is solved using Mezzarobba's implementation of numerical analytic continuation~\cite{Mezzarobba_2016} to arbitrary precision with rigorous error bounds.

The reconstruction of integer relations between transcendental numbers that are only approximately given is not possible in general. However, when the transcendental numbers are well behaved, this reconstruction may be possible. We devote \S\ref{sec:numer-reconstr-picar} to the study of this problem.

\subsection{An example}\label{sec:an_example}

Consider the quartic surface $X \subset \ppp$ defined by the polynomial
\begin{equation}
  f= 3x^3z - 2x^2y^2 + xz^3 - 8y^4 - 8w^4.
\end{equation}
As described above, we may compute a $1\times 22$ matrix of periods to arbitrary precision. In this example, the differential equations that arise are of order $5$ with polynomial coefficients of degree at most $59$. On a laptop, the determination of this differential equation takes about two seconds and it takes $30$ seconds to integrate it to $100$ digits of accuracy, with rigorous error bounds. All of this can be done with the command {\tt PeriodHomotopy([f])} using the package {\tt period-suite} written for {\tt Magma} \cite{BosmaCannonPlayoust_1997} and utilizing~{\tt ore\_algebra-analytic}~\cite{Mezzarobba_2016}.

\begin{figure}[t]
  \centering
  \begin{equation*}
    \setcounter{MaxMatrixCols}{30}
    \def\minusshort{\scalebox{0.75}[1.0]{$-$}}
    \tiny
    \arraycolsep=.4pt\def\arraystretch{1}
    \left[\begin{array}{rrrrrrrrrrrrrrrrr@{\hspace{1em}}r}
      \phantom{\minusshort}0 & \phantom{\minusshort}0 & \phantom{\minusshort}0 & \phantom{\minusshort}0 & \phantom{\minusshort}0 & \phantom{\minusshort}0 & \phantom{\minusshort}0 & \phantom{\minusshort}0 & \phantom{\minusshort}0 & \phantom{\minusshort}0 & \phantom{\minusshort}0 & \phantom{\minusshort}0 & \phantom{\minusshort}0 & \phantom{\minusshort}0 & \minusshort1669083212117905913652734 & \phantom{\minusshort}0 & 1937019641160560221317687 & \ldots \\
            0 & 0 & 0 & 0 & 0 & 0 & 0 & 0 & 0 & 0 & 0 & 0 & 0 & 0 & 0 & 1669083212117905913652734 & 1937019641160560221317687 & \ldots \\[-.4em]
            \multicolumn{18}{r}{\hrulefill}\\ 
      1 & 0 & 0 & \minusshort1 & 0 & 0 & 0 & 1 & 1 & 0 & 0 & 0 & 0 & 0 & \minusshort146511829901195443671789 & 84478429044587822467823 & \minusshort365980228690630104919296 & \ldots \\ 
      0 & 0 & 0 & 0 & 1 & 0 & 0 & 0 & 0 & 0 & 0 & 0 & 0 & 0 & \minusshort337167720252678310258177 & 224110151973403946221421 & \minusshort743116955936487279910552 & \ldots \\
      0 & 0 & 0 & 0 & 0 & 0 & 0 & 0 & 0 & 0 & 0 & 0 & 1 & \minusshort1 & 357031479253522311483650 & 768066337666351099432748 & 940525994719391079998435 & \ldots \\
      0 & 0 & 0 & 0 & 0 & 1 & 0 & 0 & 1 & 0 & 1 & 0 & 0 & 0 & \minusshort552756671828854153114905 & \minusshort126018248279583585486071 & 535095811953165917210863 & \ldots \\
      0 & \minusshort1 & 1 & 0 & 0 & 0 & 0 & 0 & 1 & 0 & 0 & \minusshort1 & 0 & 0 & 104335431129908645825133 & \minusshort231616284585318363570849 & 502730408585962411025306 & \ldots \\
      0 & 0 & 0 & 0 & 0 & 0 & 0 & 0 & 0 & 0 & 0 & 0 & 0 & \minusshort1 & \minusshort649159586430203173692632 & 770784867967071100945665 & \minusshort2152014469737999315531272 & \ldots \\
      0 & 0 & 0 & 0 & 0 & 0 & 0 & 0 & 0 & 1 & 1 & 0 & 0 & 0 & 277747983934797690835205 & \minusshort28625739873061372966384 & \minusshort638732179408358479990097 & \ldots \\
      1 & 0 & 0 & 0 & 0 & 0 & 0 & 0 & 0 & 0 & 0 & 1 & 0 & 0 & 146511829901195443671790 & \minusshort84478429044587822467823 & 365980228690630104919296 & \ldots \\
      0 & 0 & 0 & 0 & 0 & 0 & 0 & 0 & 0 & 0 & 0 & \minusshort1 & 1 & 1 & 250899146775406645936761 & 575615030011256031395007 & \minusshort114830012426104078247291 & \ldots \\
      0 & 1 & 0 & 0 & 0 & 0 & 0 & 1 & 0 & 0 & \minusshort1 & 0 & 0 & 0 & 104335431129908645825133 & \minusshort231616284585318363570849 & 502730408585962411025307 & \ldots \\
      0 & 0 & 0 & 0 & 0 & 0 & \minusshort1 & 0 & 0 & 0 & 0 & 0 & 1 & \minusshort1 & \minusshort140644950443454586919439 & \minusshort393058206212350140614235 & 429933080833930208291557 & \ldots \\
      0 & 0 & 0 & 0 & 0 & 0 & 0 & 0 & 1 & 0 & 0 & 0 & 0 & 0 & 594933070600140950961561 & 273156103820314126589096 & \minusshort671845991848498223316874 & \ldots \\
      0 & 0 & 0 & 0 & 1 & 0 & 0 & \minusshort1 & 0 & 0 & 0 & 0 & 0 & 0 & 337167720252678310258177 & \minusshort224110151973403946221421 & 743116955936487279910552 & \ldots \\
      0 & 0 & 0 & 0 & 0 & 0 & 0 & 0 & 0 & 0 & 0 & 0 & 0 & 1 & \minusshort824317154838996681984621 & 177119763197465887754938 & \minusshort236792300924643740702432 & \ldots \\
      0 & 0 & 0 & 0 & 0 & 0 & 0 & 1 & 0 & 0 & 1 & 0 & 0 & 0 & 379344119023965108104833 & \minusshort76972296432673405118395 & 606366776041154973804541 & \ldots \\
      0 & 0 & 0 & 0 & 0 & 1 & 0 & 0 & 0 & 0 & 0 & 0 & 0 & 0 & 552756671828854153114905 & 126018248279583585486070 & \minusshort535095811953165917210864 & \ldots \\
      0 & 0 & 0 & 0 & 0 & 0 & 1 & 0 & 0 & 0 & 0 & 0 & 0 & \minusshort1 & \minusshort140644950443454586919440 & \minusshort393058206212350140614234 & 429933080833930208291557 & \ldots \\
      0 & 0 & 1 & 0 & 0 & 0 & 0 & 0 & 0 & 0 & 0 & 0 & 0 & 0 & \minusshort104335431129908645825133 & 231616284585318363570849 & \minusshort502730408585962411025307 & \ldots \\
      0 & 0 & 0 & 0 & 0 & 0 & 0 & 0 & 0 & 0 & 0 & 0 & 1 & 0 & \minusshort467285675585474370500971 & \minusshort950623161465256990213520 & \minusshort1255629063127217210042702 & \ldots \\
      0 & 0 & 0 & 1 & 0 & 0 & 0 & 0 & 0 & 0 & 0 & 0 & 0 & 0 & \minusshort146511829901195443671790 & 84478429044587822467823 & \minusshort365980228690630104919296 & \ldots \\
      0 & 0 & 0 & 0 & 0 & 0 & 0 & 0 & 0 & 1 & 0 & \minusshort1 & 0 & 0 & \minusshort277747983934797690835206 & 28625739873061372966384 & 638732179408358479990097 & \ldots \\
      0 & 0 & 0 & 0 & 0 & 0 & 0 & 0 & 0 & 0 & 0 & 1 & 0 & 0 & \minusshort69025235930677842745100 & 457102914343586863258366 & 660652346877586707848817 & \ldots \\
\end{array}\right]
  \end{equation*}
  \caption{Lattice of integer relations between approximate periods. The last $5$ columns are omitted.}
 \label{fig:candidate-integer-rels} 
\end{figure}

Applying the LLL algorithm to these approximate periods of $X$ gives a basis of integer relations between the approximate periods. More precisely, we consider for the lattice $\Lambda$ of integer vectors
$(u, v, a_1,\dotsc,a_{22}) \in \Zz^{24}$ satisfying
\begin{equation}\label{eq:lattice_relation}
\sum_{i=1}^{m} a_i \left[ 10^{100} \int_{\gamma_i} \omega_X \right] = u + v\sqrt{-1},
\end{equation}
where~$[-]$ denotes the rounding to nearest integer. Equation \eqref{eq:lattice_relation} should be compared with \eqref{eq:1}.
Short vectors in~$\pic(X)$ give rise to short vectors in~$\Lambda$, and a short vector in~$\Lambda$ is likely to come from a vector in~$\pic(X)$, unless a suprising numerical cancellation happens.

Concerning the example, Figure~\ref{fig:candidate-integer-rels} shows a matrix whose columns form a LLL-reduced basis for the lattice $\Lambda$.
We observe an important gap in size, between the 14th and 15th column.
We conclude that the Picard number of~$X$ is most likely~$14$ and that the columns of the lower left $22\times 14$ submatrix is a basis of~$\pic(X)$.
The norm of the first dismissed column, about~$10^{25}$, fits precisely the expected situation described in Proposition~\ref{prop:heuristic}.

This numerical approach may fail in two ways: either by missing a relation or by returning a false relation which nevertheless holds up to high precision. Proposition~\ref{prop:gap} quantifies the way in which such a failure may occur: either the computation of~$\pic(X)$ is correct; or~$\pic(X)$ is not generated by elements of norm $<10^{20}$; or there is some $(a_i) \in\Zz^m$ with $\sum_i a_i^2 \leq 4$ such that $\big|\sum_{i} a_i \int_{\gamma_i} \omega_X\big|$ is not zero but smaller than~$\smash{10^{-99}}$.
Section~\ref{sec:intrinsic_error} expresses these quantities in terms of an intrinsic norm that is independent from the coordinates that were used to carry out the computations.

The intersection product on $\pic(X)$ is readily computed from the generators, as we now describe. The basis of homology on $X$ is obtained by carrying a basis from the Fermat surface $Y$ by parallel transport. On $Y$ the intersection numbers~$\gamma_i \cdot \gamma_j$ are known exactly as well as the polarization, i.e., the coordinates of the homology class of a general place section~$H\cap X$ in the basis~$\{\gamma_i\}_{i=1}^{22}$, see \S\ref{sec:completing_homology}. As these values remain constant during parallel transport, we know the intersection product on the homology of $X$ as well as the polarization. Computing the intersection product of the $14$ generators of~$\pic(X)$ in homology, we obtain the intersection product on $\pic(X)$. Since the polarization lies in $\pic(X)$ we express it in terms of these generators of $\pic(X)$. The result of this operation is displayed in Figure~\ref{fig:candidate_intersection_product}. The command {\tt HodgeLattice} of {\tt period-suite} performs all the operations starting from the computation of the periods.  

\begin{figure}[t]
  \small
  \begin{align*}
    \left[\begin{array}{rrrrrrrrrrrrrr}
            -4 & 0  & 0  & 2  & 2  & 2  & 0  & -3 & -1 & -2 & 0   & -1 & 1  & -1 \\
            0  & -4 & 2  & 0  & -1 & 1  & 2  & -1 & 2  & 0  & 4   & -2 & 0  & 0 \\
            0  & 2  & -4 & 0  & 2  & 0  & -2 & -1 & -1 & 0  & 0   & 2  & 1  & -1 \\
            2  & 0  & 0  & -4 & -1 & -1 & 0  & 3  & 2  & 0  & 0   & 0  & 0  & 0 \\
            2  & -1 & 2  & -1 & -4 & 0  & 1  & 3  & 1  & 2  & 1   & 0  & -1 & -1 \\
            2  & 1  & 0  & -1 & 0  & -4 & -1 & 1  & -1 & 2  & -3  & -1 & 1  & 3 \\
            0  & 2  & -2 & 0  & 1  & -1 & -4 & 1  & -2 & 0  & -2  & 0  & 2  & 2 \\
            -3 & -1 & -1 & 3  & 3  & 1  & 1  & -6 & -1 & -1 & 1   & 0  & 1  & -1 \\
            -1 & 2  & -1 & 2  & 1  & -1 & -2 & -1 & -4 & 1  & -2  & 0  & 2  & 0 \\
            -2 & 0  & 0  & 0  & 2  & 2  & 0  & -1 & 1  & -4 & 0   & 1  & -1 & 1 \\
            0  & 4  & 0  & 0  & 1  & -3 & -2 & 1  & -2 & 0  & -10 & -1 & 0  & 3 \\
            -1 & -2 & 2  & 0  & 0  & -1 & 0  & 0  & 0  & 1  & -1  & -6 & 2  & 3 \\
            1  & 0  & 1  & 0  & -1 & 1  & 2  & 1  & 2  & -1 & 0   & 2  & -4 & 0 \\
            -1 & 0  & -1 & 0  & -1 & 3  & 2  & -1 & 0  & 1  & 3   & 3  & 0  & -10\\
          \end{array}\right]
    \quad
    \left[\begin{array}{r}
            -4 \\
            -5 \\
            0 \\
            -2 \\
            4 \\
            3 \\
            1 \\
            3 \\
            -1 \\
            6 \\
            -2 \\
            4 \\
            0 \\
            2
          \end{array}\right]
  \end{align*}
  \caption{Matrix of the intersection product and the
    coordinates of the hyperplane section in $\pic(X)$.}
  \label{fig:candidate_intersection_product} 
\end{figure}

Applying standard methods to be discussed in \S\ref{sec:P1inK3}, we find from the Picard lattice of $X$ that there are $4$~lines, $102$~quartic curves and no twisted cubics inside $X$.

\subsection{Transcendental lattice and reduction to finite characteristic}
\label{sec:transcendental}

\paragraph{Definition and properties}

Let $X$ be a quartic surface. Beyond the Picard group of $X$, we can compute its \emph{transcendental lattice} and its \emph{endomorphism ring}.
The transcendental lattice of $X$ is defined as
\begin{equation}
T = \left\{ \omega \in H^2(X,\Qq) \st \forall \gamma\in\pic(X), \int_\gamma \omega = 0  \right\},
\end{equation}
which is a $\Qq$-linear space of dimension~$22-\rk \pic(X)$.
Furthermore, let $T_\Cc = T \otimes \Cc \subset H^2(X, \Cc)$ and observe $\omega_X \in T_\Cc$ by~\eqref{eq:1}.
The endomorphism ring~$E \subset \operatorname{End}_\Qq(H^2(X, \Qq))$ is defined as the subring of all linear maps~$e$ such that
$e(\omega_X) \in \Cc \langle \omega_X \rangle$, where~$e$ has been extended canonically to~$H^2(X, \Cc)$.
The map~$\phi \colon E\to \Cc$ defined by~$e(\omega_X) = \phi(e) \omega_X$ is an injective ring morphism and every element in $E$ is invertible, therefore~$E$ is a number field~\cite[Corollary~3.3.6]{huybrechts--k3}.
In fact, $E$ is either totally real or a CM-field~\cite{zarhin--k3}, see also~\cite{huybrechts--k3}.

Charles~\cite{Charles_2014} determined in terms of~$E$ the overestimation of reduction methods to compute the Picard number of K3 surfaces.
We give here a quick overview, see \cite{huybrechts--k3,tretkoff--book} for further results. Although we state these results for quartic surfaces over~$\Qq$ much of it holds for any K3 surface over a number field.

If the quartic~$X \subset \ppp$ is defined by a polynomial~$f$ with integer coefficients,
we may consider for all but finitely many prime~$p$ the smooth quartic surface~$X_p$ defined over~$\Ff_p$ by the reduction of~$f$ modulo~$p$. Let~$\rho$ and~$\rho_p$ denote the (geometric) Picard numbers of~$X$ and~$X_p$ respectively.
Let $\rho_{\red}$ be the minimum of the set $\left\{\rho_p \st p \text{ prime and } X_p \text{ smooth}\right\}$.
The starting point of reduction methods is the inequality~$\rho \leq \rho_\red$ and the relative ease with which the numbers $\rho_p$ are computed. A key issue is to determine whether~$\rho = \rho_\red$.

Although $\rho$ can be either even or odd, $\rho_p$ is always even. This issue was partially overcome by van~Luijk~\cite{vanLuijk_2007} who gave necessary conditions for $\rho = \rho_{\red}$. He used his argument to give the first example of a K3 surface defined over the rationals with Picard number~$1$ by exhibiting a surface~$X$ with~$\rho_\red = 2$ that does not satisfy his necessary condition. It was asked by Elsenhans and Jahnel~\cite{elsenhans_jahnel--kummer} whether $\rho=\rho_{\red}$ if $\rho$ is even and $\rho=\rho_{\red}-1$ if $\rho$ is odd. Charles~\cite{Charles_2014} settled the question in the negative.

\begin{theorem}[Charles]\label{thm:charles}
  The equality $\rho_{\red}=\rho$ holds unless $E$ is totally real and the dimension of $T$ over $E$ is odd, in which case $\rho_{\red}=\rho + \dim_{\Qq} E$.
\end{theorem}

\paragraph{Computation}
 From the numerical computation of periods, we obtain approximations of~$a_i \eqdef \int_{\gamma_i} \omega_X \in \Cc$ for some basis~$\gamma_1,\dotsc,\gamma_{22}$ of~$\H_2(X,\Zz)$.
The cohomology group~$H^2(X,\Cc)$ is endowed with the dual basis~$\gamma_1^*,\dotsc,\gamma_{22}^*$ so that~$\omega_X = \sum_i a_i \gamma_i^*$.
Once a basis~$u_1,\dotsc,u_\rho$ of the Picard group~$\pic(X)$ is computed, a basis~$v_1,\dotsc,v_{\rho'}$ of~$T \eqdef \pic(X)^\perp \subseteq \H^2(X,\Qq)$ is found readily.

For~$e \in \operatorname{End}_\Qq(H^2(X, \Qq))$ the condition $e \in E$ can be rewritten as:
\begin{equation}
\exists \lambda \in \Cc: e(\omega_X) = \lambda \omega_X \Leftrightarrow \langle \omega_X, e(\omega_X) \rangle \omega_X = \langle \omega_X,\omega_X \rangle e(\omega_X).
\end{equation}
Writing $A = (a_1,\dotsc,a_{22})^t$ for the coefficient vector of~$\omega_X$, we can compute the endomorphism ring~$E$ via the following formulation:
\begin{equation}
  E = \Qq \cdot \left\{ M \in \Zz^{\rho'\times\rho'} \st (\bar A^t M A) A = (\bar A^t A) MA \right\}.
\end{equation}
Just as with the computation of~$\pic(X)$, the problem of computing $E$ is now a problem of computing integer solutions to linear equations with approximate real coefficients. We approach it once again with lattice reduction algorithms, see~\S \ref{sec:numer-reconstr-picar}. Examples are provided in~\S\ref{sec:experiment}.

\section{Smooth rational curves in K3 surfaces} \label{sec:P1inK3}

The data  of the matrix of
the intersection product in some basis of the Picard group of a
smooth quartic surface $X \subset \ppp$ together with the coordinates of the
class of hyperplace section in the same basis (as in Figure~\ref{fig:candidate_intersection_product}) is enough to count all smooth rational
curves of a given degree lying on~$X$.
We describe an algorithm here that, at least in broad strokes, seems to be folklore.\footnote{We are indebted to Alex Degtyarev for sharing his understanding with us.}

In principle, smooth rational curves in a surface can be enumerated using purely symbolic methods and for lines this process is routine. However, it is a challenge to enumerate even the quadric curves in quartic surfaces, let alone higher degree curves in higher degree surfaces. The computation of the Picard group offers an indirect solution to this problem.

Fix a smooth quartic~$X\subset\ppp$ and for each positive integer $d$ let~$\cR_d$ be the set of all smooth rational curves of degree~$d$ lying in~$X$. In order to compute the cardinality of the set $\cR_d$ we will first observe that a smooth rational curve is completely determined by its linear equivalence class. Recall that we denote by~$h_X \in \pic(X)$ the class of a hyperplane section. For $d > 0$ we define the set~$M_d = \left\{ D \in \pic(X) \st D^2 = -2,\ D\cdot h_X = d\right\}$.

\begin{lemma}\label{prop:Nd-injection}
  A smooth rational curve in $X$ is isolated in its linear equivalence class. Moreover, the map~$\cR_d \to \pic(X)$ which maps a rational curve to its linear equivalence class injects $\cR_d$ into $M_d$. 
\end{lemma}

\begin{proof}
  Let~$C \in \cR_d$ and~$D=[C] \in \pic(X)$. As $C$ is of degree $d$, it intersects a general hyperplane in $d$ points so that $C \cdot h_X = d$. Recall that the canonical class $K_X$ of the K3 surface $X$ is trivial so that adjunction formula reads $D^2 = D\cdot(K_X+D)=2g(\p)-2=-2$ \cite[\S II.11]{BarthHulekPetersEtAl_2004}\cite[Ex.~V.1.3]{Hartshorne_1977}. This proves that the image of $\cR_d$ lies in $M_d$.

  Now we show that $C$ is isolated in its linear system. Indeed, if $C'$ is a curve linearly equivalent to but different from $C$, then the intersection number $[C] \cdot [C']$ must be positive, as this number can be obtained by counting the points in $C \cap C'$ with multiplicity. This leads to a contradiction: $-2 = [C]^2 = [C]\cdot[C'] > 0$.
\end{proof}

Typically, the inclusion $\cR_d \toi M_d$ is strict. We now demonstrate that with knowledge of $M_{d'}$ for each $d' \le d$ one can compute the image of $\cR_d$ in $M_d$. For each $d > 0$ define inductively a subset $N_d \subset M_d$ as follows:
\begin{equation}
  N_d = \left\{ D \in M_d \st \forall d'< d, \forall D'\in N_{d'}, D' \cdot D \geq 0 \right\}.
\end{equation}
Note that when $d=1$ there are no constraints and we have $N_1 = M_1$.

\begin{proposition}\label{prop:image_of_Rd}
  For $d >0$, the image of the inclusion $\cR_d \toi M_d$ is a bijection onto $N_d$.
\end{proposition}

\begin{proof}
  For any two distinct irreducible curves $C$ and~$C'$ we have $C \cdot C' \ge 0$. Upon taking $C \in \cR_d$ and $C' \in \cR_{d'}$ for $d' < d$ we see that $\cR_d$ injects in to $N_d$. 
  
  Now take any~$D \in N_d$. From the Riemann--Roch theorem for surfaces \cite[V.1.6]{Hartshorne_1977} we get:
  \begin{equation*}
     \dim \H^0(X, \co_X(D)) + \dim \H^0(X, \co_X(-D)) \geq \frac12 D^2 + 2 = 1,
  \end{equation*}
  so that either $D$ or~$-D$ must be effective. Since~$D\cdot h_X > 0$, $-D$ can not be effective and therefore $D$ must be.

  Let us write~$D$ as a sum of classes of distinct irreducible curves~$\sum_i n_i C_i$ with $n_i >0$. Since $D^2 < 0$ there exists an index $i$ such that $C_{i} \cdot D <0$. Moreover, $C_{i} \cdot C_j \geq 0$ for every~$j\neq i$, so~$C_{i}^2 < 0$. By adjunction formula, we conclude that $C_{i}$ must be a smooth rational curve \cite[Ex.~IV.1.8]{Hartshorne_1977}. Furthermore, let $d'=C_{i} \cdot h_X$ and observe $d' \le d$. By definition of $N_d$ we must have $d'=d$ and therefore $D = C_{i}$. Therefore, $\cR_d$ surjects onto $N_d$.
\end{proof}

Proposition~\ref{prop:image_of_Rd} implies that in order to compute the cardinality of the set $\cR_d$ it suffices to compute the set $N_d$ (see Algorithm~\ref{alg:count_curves}). The latter can be easily computed from the sets $M_{d'}$ for $d' \le d$. We now reduce the computation of $M_d$ for each $d > 0$ to the enumeration of all vectors of a given norm in a lattice with a negative definite quadratic form.

Let~$\pic^0(X) = \left\{ D\in\pic(X) \st D\cdot h_X = 0 \right\}$.  The intersection product on $\pic^0(X)$ is negative definite~\cite[Proposition~1.2.4]{huybrechts--k3}. Recalling that $h_X^2=4$, we define a map 
$\pi \colon \pic(X) \to \pic^0(X)$ with $ \pi(D)= 4D - (D\cdot h_X) h_X$. 

The map $\pi$ maps $M_d$ bijectively on to the following set:
  \begin{equation}
    \overline{M}_d =  \left\{ E \in \pic^0(X) \st E^2 = -(32 + 4 d^2) \text{ and } E + d h_X \in 4 \pic(X) \right\}.
  \end{equation}
The inverse map $\overline{M}_d \to M_d$ is given by $E \mapsto \frac{1}{4} (E+dh_X)$.

In order to compute $\overline{M}_d$ we first find the finitely many elements $E\in\pic^0(X)$ of norm~$-(32 + 4d^2)$, for example using KFP algorithm \cite{Kannan_1983,FinckePohst_1983}. Then, among all such $E$, we select those where $\frac14 \left( E + dh_X \right)$ has integer coordinates to obtain $\overline{M}_d$. 
  In practice, it is sufficient and  more efficient to enumerate the elements of length~$-(32 + 4d^2)$ in the sublattice~$\pi(\pic(X)) = \pic^0(X) \cap (4 \pic(X) + \Zz h_X)$.

\begin{algorithm}[t] \centering
  \begin{description}[leftmargin=5em,style=nextline]
    \item[Input.]
           The Picard group (i.e. matrix of the interction product in some basis and the coordinates of the class of hyperplane section in the same basis) of a smooth
quartic surface $X \subset \ppp$;
an integer~$d > 0$.
  \item[Output.] The set $\left\{ [C] \in \pic(X) \st \text{$C\subset X$ is
        a smooth rational curve of degree~$d$} \right\}$.
  \end{description} \medskip
  \begin{algorithmic}
    \Function{RationalCurves}{$\pic(X)$, $d$}
    \State Compute a basis of~$\pic^0(X) = \left\{ D\st D\cdot h_X =
      0 \right\} \subset \pic(X) \simeq \Zz^\rho$
    \State Compute a basis of~$4 \pic(X) + \Zz h_X \subset \pic(X)$
    \State Compute a basis of~$\Lambda = \pic^0(X) \cap (4 \pic(X) + \Zz h_X)
    \subset \pic(X)$
    \State $S \gets \left\{ D \in \Lambda \st -D^2 = 32 + 4d^2 \right\}$
    \Comment{e.g.
    with KFP algorithm \cite{Kannan_1983,FinckePohst_1983}}
    \State $M_d \gets \left\{ \frac14 (D+d h_X) \st D \in S \right\} \cap \Zz^\rho$
    \State \textbf{return} $\left\{ D \in M_d \st \forall d' < d, \forall
      D'\in\operatorname{\textsc{RationalCurves}}(\pic(X), d'), D\cdot D' \geq 0 \right\}$
    \EndFunction
  \end{algorithmic}
  \caption{Finding rational curve classes on smooth quartic surfaces.}
  \label{alg:count_curves}
\end{algorithm}

\begin{example}
  Take $f_X = 14x^4 - 85x^3z - 2xz^3 + 83y^4 - 17y^3w - 96w^4$ and let $X=Z(f_X) \subset \ppp$. We find that $X$ has Picard number $18$ with the following representation of $(\pic(X),h_X)$:
\begin{align*}
  \tiny
    \left[
    \begin{array}{rrrrrrrrrrrrrrrrrr}
    -4 & 0  & -1 & 0  & -1 & 2  & 0  & 0  & 0  & -2 & 0  & -2 & 0  & 1  & -1 & -1 & 0  & -2 \\
     0 & -6 & -3 & -3 & -3 & 3  & 0  & -2 & -2 & -1 & 3  & -1 & -1 & 3  & -1 & 0  & 1  & 0 \\
    -1 & -3 & -4 & -2 & -2 & 2  & 0  & 0  & -1 & 0  & 2  & 0  & -2 & 3  & -2 & 0  & 1  & 0 \\
     0 & -3 & -2 & -4 & -2 & 2  & 1  & -1 & 0  & -1 & 1  & -1 & -1 & 2  & -1 & 0  & 1  & 0 \\
    -1 & -3 & -2 & -2 & -4 & 2  & 0  & 0  & -1 & -1 & 2  & 0  & 0  & 2  & -2 & 0  & 1  & 0 \\
     2 & 3  & 2  & 2  & 2  & -4 & 1  & 2  & 2  & 2  & -1 & 2  & 0  & -2 & 2  & 0  & -1 & 2 \\
     0 & 0  & 0  & 1  & 0  & 1  & -4 & -2 & -2 & 1  & 1  & 1  & 0  & 0  & -1 & 2  & 1  & 1 \\
     0 & -2 & 0  & -1 & 0  & 2  & -2 & -4 & -2 & 0  & 1  & -1 & 0  & 1  & 0  & 1  & 1  & 0 \\
     0 & -2 & -1 & 0  & -1 & 2  & -2 & -2 & -4 & 0  & 2  & 0  & 0  & 1  & -1 & 1  & 2  & 0 \\
    -2 & -1 & 0  & -1 & -1 & 2  & 1  & 0  & 0  & -4 & 0  & -2 & 1  & 0  & 0  & -2 & -1 & -2 \\
     0 & 3  & 2  & 1  & 2  & -1 & 1  & 1  & 2  & 0  & -4 & 0  & 2  & -2 & 1  & -1 & -2 & -2 \\
    -2 & -1 & 0  & -1 & 0  & 2  & 1  & -1 & 0  & -2 & 0  & -4 & 0  & 0  & 0  & 0  & 0  & -2 \\
     0 & -1 & -2 & -1 & 0  & 0  & 0  & 0  & 0  & 1  & 2  & 0  & -4 & 3  & -1 & 1  & 2  & 2 \\
     1 & 3  & 3  & 2  & 2  & -2 & 0  & 1  & 1  & 0  & -2 & 0  & 3  & -6 & 3  & 0  & -2 & 0 \\
    -1 & -1 & -2 & -1 & -2 & 2  & -1 & 0  & -1 & 0  & 1  & 0  & -1 & 3  & -4 & 1  & 2  & -1 \\
    -1 & 0  & 0  & 0  & 0  & 0  & 2  & 1  & 1  & -2 & -1 & 0  & 1  & 0  & 1  & -4 & -2 & -1 \\
     0 & 1  & 1  & 1  & 1  & -1 & 1  & 1  & 2  & -1 & -2 & 0  & 2  & -2 & 2  & -2 & -4 & -1 \\
    -2 & 0  & 0  & 0  & 0  & 2  & 1  & 0  & 0  & -2 & -2 & -2 & 2  & 0  & -1 & -1 & -1 & -4
    \end{array}\right]
  \tiny
    \left[
    \begin{array}{r}
  -2\\ -1\\ 1\\ 1\\ 3\\ 4\\ 4\\ 0\\ 0\\ 0\\ 0\\ 2\\ 2\\ 0\\ -2\\ 2\\ -1\\ 4 
    \end{array}\right].
\end{align*}
Applying Algorithm~\ref{alg:count_curves} we see that there are 16 lines, 288 quadrics and 1536 twisted cubics as determined by this lattice of $X$. The 16 lines, and their incidence correspondence, as we compute from this lattice are in agreement with what we can compute rigorously using symbolic methods.
\end{example}

\section{Numerical reconstruction of integer relations}
\label{sec:numer-reconstr-picar}

In view of~\eqref{eq:1} and~\eqref{eq:2}, recovering $\pic(X)$ boils down to
finding integer linear relations between the period vectors.
With the methods employed here, a finite but high enough precision will successfully recover~$\pic(X)$.
It seems difficult to decide if a given precision is ``high enough''.
Instead, we will study the process of finding linear relations between approximate vectors of real numbers and quantify the expected behavior of ``noise'', that is, of relations that are an artifact of the finite approximation. We will thus select relations whose behavior significantly differs from the expected behavior of noise.

The reconstruction of integer relations between real numbers is a well known
application of the Lenstra--Lenstra--Lovász lattice basis reduction
algorithm~\cite[p.~525]{LenstraLenstraLovasz_1982}, see also~\cite{BuchmannPohst_1989,ChenStehleVillard_2013}.
There are many other algorithms for the problem of computing integer
relations, in particular the HJLS~\cite{HastadJustLagariasEtAl_1989} and
PSLQ~\cite{FergusonBaileyArno_1999,Bailey_2008,FengChenWu_2018} families
and the first successful algorithm by Ferguson and
Forcade~\cite{FergusonForcade_1979}. 
A strong point in favor of the folklore LLL approach
is that efficient LLL implementations are available in most computer algebra systems.
To the best of our knowledge, existing work do not address the problem of
computing the full lattice of integer relations (not just one) between real
vectors (not just real numbers) given by approximations (not assuming exact data
and exact arithmetic).

In this section, we recall and analyze the LLL approach to solve the following problem:
\emph{Given a numerical approximation of a real matrix $P\in \Rr^{m \times p}$, with $p\leq m$,
recover a basis of the lattice $\Lambda = \left\{ x \in \Zz^{m} \st x P = 0  \right\}$.}

In our setting, the coefficients of~$P$ are the real and imaginary parts
of the periods~$\int_{\gamma_i} \omega_j$ of the surface~$X$ under
consideration.
For~$B>0$ and~$\varepsilon > 0$ let~$\Lambda_{B,\varepsilon}$
be the lattice
\begin{equation}
  \Lambda_{B,\varepsilon} = \left\langle x\in \Zz^m \st \|x\| \leq B \text{ and }
    \|x P \| < \varepsilon \right\rangle.
\end{equation}

A rigorous numerical computation of~$\Lambda$ faces two obstacles: the lack of an
\emph{a priori} bound on the norm of generators and the inability to recognize
zero among periods. In contrast, the lattice~$\Lambda_{B,\varepsilon}$ can be
computed exactly for given large~$B$ and small~$\varepsilon$.

If~$B$ is larger than the length of the largest vector in a generating family of~$\Lambda$, then for all~$\varepsilon>0$
small enough~$\Lambda_{B,\varepsilon} = \Lambda$.
No \emph{a priori} bounds are known about the values of~$B$
and~$\varepsilon$ that would ensure the desired equality $\Lambda_{B,\varepsilon} =
\Lambda$. If we choose~$B$ too small, we may miss integer relations in~$\Lambda$. If we
choose~$\varepsilon$ too big, we may compute relations that do not belong to~$\Lambda$.
Yet, meaningful results can be obtained by comparing with the \emph{expected} situation.

Assume that, for some large~$\beta > 0$ (typically $10^{300}$), we are given the exact value of the
$m\times p$  integer matrix $P_\beta$ obtained by entry wise rounding to the nearest integer the
coefficients of $\beta P$, that is
\begin{equation}
P_\beta = \beta P + E, \quad\text{with } E \in [-\tfrac12,\tfrac12]^{p\times m}.
\end{equation}
Then, we build the $m\times (p+m)$ integer matrix $M = \left[ \begin{array}{c|c} P_\beta & I_m \end{array} \right]$ and compute an LLL-reduced basis $b_1,\dotsc,b_m$ of the lattice spanned by the rows of~$M$.

We complement the folklore LLL approach with the following heuristic.
If~$\beta$ is large enough, Proposition~\ref{prop:heuristic} suggests that for $\rho=\rk \Lambda$ the norm
$\|b_\rho\|$ is small but the norm~$\|b_{\rho+1}\|$ is large and comparable to $\beta^{\frac p{m-\rho}}$.
In this case, $\Lambda = \langle \pr(b_1),\dotsc,\pr(b_\rho) \rangle$, where $\pr \colon \Zz^{p+m}\to \Zz^m$ is the projection on to the last $m$ coordinates.

\begin{algorithm}[t] \centering
  \begin{description}[leftmargin=8em,style=nextline]
  \item[Input.] $Q \in \Zz^{p\times m}$ and $\beta > 0$.
    \item[Precondition.] $Q = \beta P + E$ for some~$P \in \Rr^{p\times m}$ and~$E\in [-\tfrac12,\tfrac12]^{p\times m}$.
    \item[Output.] Fail or return~$u_1,\dotsc,u_\rho \subset \Zz^m$ and $B, \epsilon > 0$.
    \item[Postcondition.] Either $\langle u_1,\dotsc,u_\rho \rangle = \left\{ x \in \Zz^{m} \st x P = 0  \right\}$;\\
      or~$ \left\{ x \in \Zz^{m} \st x P = 0  \right\}$ is not generated by vectors of norm at most~$B$;\\
      or~$\exists x\in \Zz^m\colon \|x\| \leq \|u_\rho\|$, $\| xP\| \leq \epsilon$ and~$xP \neq 0$.
  \end{description} \medskip
  \begin{algorithmic}
    \Function{IntegerRelationLattice}{$Q$, $\beta$}
    \State Compute an LLL-reduced basis~$b_1,\dotsc,b_m$ of the lattice spanned by the rows~$\left[ \begin{array}{c|c} Q & I_m \end{array} \right]$.
    \State Find $\rho$ such that~$\|b_\rho\| \leq 2^{-m} \|b_{\rho+1}\|$ and~$\beta^{\frac p{m-\rho}}
    \approx \|b_{\rho+1}\|$. Fail if there is none.
    \State $\Lambda \gets \langle \pr(b_1),\dotsc,\pr(b_\rho) \rangle$, where $\pr \colon \Zz^{p+m}\to \Zz^m$ takes the last $m$ coordinates.
    \State $B \gets \frac1m 2^{-\frac{m+1}{2}} \| b_{\rho+1} \|$
    \State $\epsilon \gets m \beta^{-1} \|b_\rho\|$
    \State \textbf{Return} $(\Lambda, B, \epsilon)$
    \EndFunction
  \end{algorithmic}
  \caption{Computation of the lattice of integer relations between approximate real vectors with a heuristic check.}
  \label{algo:integer-relations}
\end{algorithm}

\subsection{Quantitative results}
\label{sec:from-an-appr}

For~$B \geq 1$, let
\begin{equation}
\varepsilon(B) = \min \left\{ \|uP\| \st u\in\Zz^m, \norm{u}
  \le B \text{ and } u P \neq 0\right\}.
\end{equation}
Equivalently, $\varepsilon(B)$ is the largest real number such
that
$\Lambda_{B,\varepsilon(B)} \subseteq \Lambda$.
Since $\varepsilon(B)$ is
non-increasing as a function of~$B$, the quotient $B/\varepsilon(B)$ is strictly
increasing as function of~$B$. In particular, for~$s \geq 0$ we may define a
non-decreasing function $\phi$ with
\begin{equation}
\phi(s)=\max\{B \geq 0 \mid mB/\varepsilon(B) \leq s \}.
\end{equation}
The growth of this function governs the ability to numerically reconstruct~$\Lambda$.

As above, assume that, for some~$\beta > 0$, we are given the exact value of the
integer $m\times p$ matrix $P_\beta$ obtained by entry wise rounding to the
nearest integer the coefficients of $\beta P$. Having coefficients
in~$[-\frac12,\frac12]$, the error matrix $E = P_\beta - \beta P$ satisfies $\| E
\|_{\mathrm{op}} \leq \frac12 \sqrt{pm} \leq m-1$, where
$\norm{\cdot}_{\mathrm{op}}$ denotes the operator norm, and where we
used~$\frac12 \sqrt{pm} \leq \frac12 m \leq m-1$, as~$m\geq 2$.

Let $R$ be the lattice generated by the rows of the integer $m\times (p+m)$ matrix $M = \left[ \begin{array}{c|c}
    P_\beta & I_m \end{array} \right]$ and let
$b_1,\dotsc,b_m$ be an LLL-reduced basis of~$R$.
We denote $B_0 = 0$ and $B_i = \|b_i\|$, for~$1\leq i\leq
m$.
In particular $B_0 \leq B_1 \leq \dotsb \leq B_{m}$.                                                
Gaps in this sequence typically separate the elements of~$R$ that come from genuine integer
relations from spurious relations coming from the inaccuracy of the approximations.
  
\begin{proposition}\label{prop:gap}Let~$\kappa =  m^{-1} 2^{-\frac{m+1}{2}}$.
  For any~$i \in \left\{ 0,\dotsc,m-1 \right\}$ such that~$B_i \leq\kappa B_{i+1}$, at least one
  of the following propositions holds:
  \begin{enumerate}[(i)]
  \item $\{\pr(b_1),\dots,\pr(b_i) \}$ is a basis of~$\Lambda$;
  \item $\Lambda$ is not generated by elements of norm~$\leq \kappa B_{i+1}$;
  \item $\phi(\beta) \leq B_i$.
  \end{enumerate}
\end{proposition}

\begin{proof}
  By Lemma~\ref{lem:project_reduction} below, we have
  \[ \Lambda_{B_i,m\beta^{-1} B_i} = \Lambda_{\kappa B_{i+1},m\beta^{-1} \kappa B_{i+1}} = \langle \pr(b_1),\dots,\pr(b_i) \rangle. \]
  If $\Lambda$ is generated by elements of norm~$\leq \kappa B_{i+1}$
  then $\Lambda \subseteq
  \Lambda_{\kappa B_{i+1}, m\beta^{-1}\kappa B_{i+1}}$, and therefore~$\Lambda \subseteq
  \Lambda_{B_i,m\beta^{-1} B_i}$.
  If moreover~$\phi(\beta) > B_i$, then  $m \beta^{-1} B_i \leq
  \varepsilon(B_i)$, by definition of~$\phi$,
  and this implies that $\Lambda_{B_i,m\beta^{-1} B_i} \subseteq \Lambda$.
\end{proof}

\begin{lemma}\label{lem:project_reduction}
  For any~$i \in \left\{ 0,\dotsc,m-1 \right\}$ and any~$B \in [B_i, \kappa B_{i+1}]$
  \[ \Lambda_{B, m B \beta^{-1}}= \langle \pr(b_1),\dots,\pr(b_i) \rangle. \]
\end{lemma}

\begin{proof}
  Let~$\Lambda_i \subset \Zz^m$ be the lattice generated
  by~$\pr(b_1),\dotsc,\pr(b_i)$ and let $R_i\subset R$ be the lattice generated
  by $\langle b_1,\dotsc,b_i \rangle$. Let~$R|_\tau$ denote the sublattice of
  $R$ generated by vectors of length at most~$\tau$.

  We first show $\Lambda_i \subseteq \Lambda_{B, m B \beta^{-1}}$.
  Let~$x = \pr(b_j)$, with~$j \leq i$.
  We have $\norm{x} \le \norm{b_j} \leq \norm{b_i} \le B$.
  Moreover $b_j = \left[\begin{array}{c|c} xP_\beta & x \end{array} \right]$,
  so $\norm{xP_\beta} < \norm{b_j} \leq B$.
  Since $xP = \beta^{-1}(xP_\beta- x E)$, we obtain
  \begin{equation}
    \norm{xP} \le \beta^{-1} \left( \norm{xP_\beta} + (m-1) \norm{x} \right) < m B \beta^{-1}.
  \end{equation}

  Conversely, let~$x \in \Zz^m$ such that~$\|x\| \leq B$ and~$\|x P\| <
  mB\beta^{-1}$. Let~$r = \left[\begin{array}{c|c} xP_\beta & x \end{array}
  \right]$.
  We check that
  \begin{align}
    \norm{r} &\leq \norm{x P_\beta} + \norm{x} \leq \beta \norm{xP} +\norm{xE}  + \norm{x} \\
             &\leq 2mB \nonumber 
             < 2^{-(m-1)/2} B_{i+1}.\nonumber 
  \end{align}
  The properties of
  an LLL-reduced basis \cite[Thm.~9]{Nguyen_2009} imply that no family of $i+1$
  vectors in $R$ with norms less than $2^{-(m-1)/2} B_{i+1}$ is
  independent.
  Since~$b_1,\dotsc,b_i$ are independent and of norm~$\leq B$,
  it follows that~$r \in \Qq R_i$.
   Moreover~$R$ is a primitive lattice (that is~$\Qq R \cap \Zz^{p+m} = R$)
  therefore any subset of the basis~$b_1,\dotsc,b_m$ of~$R$ generates a
  primitive lattice, so $r\in R_i$.
  And therefore~$x = \pr(r) \in \Lambda_i$.
\end{proof}

The size of the gap between~$B_{\rk \Lambda}$ and~$B_{\rk \Lambda + 1}$ can be
described more precisely in terms of~$\phi(\beta)$.

\begin{proposition}\label{prop:heuristic}
  Let~$\rho = \rk\Lambda$
  and let~$C$ be the smallest real number such that~$\Lambda$
  is generated by elements of norm at most~$C$. For any~$\beta > 0$:
  \begin{enumerate}[(i)]
    \item $B_\rho \leq \tfrac12\kappa^{-1} C$;
    \item $\phi(\beta) \leq B_{\rho+1}$;
  \end{enumerate}
  Moreover, if~$C \leq 2\phi(\beta)$, then
  \begin{enumerate}[(i),resume]
  \item $\kappa B_{\rho+1} \leq \phi(\beta)$;
  \item $\pr(b_1),\dotsc,\pr(b_\rho)$ is a basis of $\Lambda$.
  \end{enumerate}
\end{proposition}

\begin{proof}
  For $x\in \Zz^m$ let $r(x) = \left[\begin{array}{c|c} xP_\beta & x \end{array}
  \right] \in R$.
  If~$x\in\Lambda$, that is~$xP = 0$,
  \begin{equation}
    \|r(x)\| \leq \|xP_\beta\| + \|x\| \leq \beta\|xP\| + \|xE\| + \|x\| \leq m \|x\|,
  \end{equation}
  using~$P_\beta = \beta P + E$.
  In particular,
  $R$ contains $\rho$ independent elements of norm at most~$m C$.
  This implies that~$B_\rho \leq m 2^{\frac{m-1}{2}} C =  \tfrac12\kappa^{-1} C$; this is~(i).

  For (ii), since~$\Lambda$ has rank~$\rho$, at least one of the
  $\pr(b_1),\dotsc,\pr(b_{\rho+1})$ is not in~$\Lambda$, say~$\pr(b_i)$, denoted~$x$.
  Since~$x = \pr(b_i) \not\in \Lambda$, $xP \neq 0$
  and~$\|x P\| \geq \varepsilon(B_{\rho+1})$.
  Moreover~$B_{\rho+1} \geq \|b_i\| \geq  \|xP_\beta\|$, because~$b_i = r(x)$.
  It follows
  \begin{equation}
    B_{\rho+1} \geq \beta\|xP\| - \|xE\| \geq \beta \varepsilon(B_{\rho+1}) - (m-1) B_{\rho+1}, 
  \end{equation}
  which implies $\phi(\beta) \leq B_{\rho+1}$.

  To check~(iii), let~$x\in\Zz^m$ such that~$\|x\|
  \leq \phi(\beta)$ and~$\|xP\| = \varepsilon(\phi(\beta))$.
  By construction, $x\not\in \Lambda$. The element~$r(x) \in R$ satisfies
  \begin{equation}
    \|r(x)\| \leq \beta \|x P\| + \|xE\| + \|x\| \leq \beta \varepsilon(\phi(\beta)) + m \phi(\beta).
  \end{equation}
  By definition of~$\phi$,  $\beta \varepsilon(\phi(\beta)) = m \phi(\beta)$
  and therefore~$\|r(x)\| \leq 2m\phi(\beta)$.
  As shown above, $R$ contains $\rho$ independent elements of norm~$\leq m C$
  that project to elements of~$\Lambda$. The vector $r(x) \in R$ does not
  project on~$\Lambda$, so $R$ contains~$\rho+1$ independent elements of norm
  $\leq m\max(C, 2\phi(\beta)) = 2\phi(\beta)$.
  It follows that~$\kappa B_{\rho+1} \leq \phi(\beta)$.
\end{proof}

Minkowski's Theorem on linear forms
shows that if~$\varepsilon^p \beta^{m-\rk \Lambda - p} \geq \det(P^T P)$,
there is an~$x \in \Lambda^\perp \cap \Zz^m$ such that~$\|xP\| \leq p \varepsilon$.
Therefore
\begin{equation}
  \varepsilon(\beta) = O\left(\beta^{1-\frac {m-\rk \Lambda}p}\right)
\quad\text{ and }\quad
  \phi(s) = O\left(s^{\frac{p}{m-\rk \Lambda}}\right).
\end{equation}

We define the \emph{irrationality measure} of~$P$, denoted~$\mu(P)$ as the infimum of all~$\mu > 0$ such that~$\varepsilon(\beta) = O(\beta^{1-\mu})$ as~$\beta\to\infty$.
As for the usual irrationality of real numbers, we can show with Borel--Cantelli Lemma that~$\mu(P) = \frac{m-\rk \Lambda}{p}$
for allmost all~$P\in
\Rr^{m\times p}$ with a given lattice~$\Lambda$ of integer relations.
Generalizing Roth's Theorem on rational approximation of algebraic numbers,
Schmidt~\cite{Schmidt_1971} proved that if~$P$ has \emph{algebraic} coefficients, with some
additional hypotheses, then it again holds that~$\mu(P) = \frac{m-\rk \Lambda}{p}$.

All in all, this leads to Algorithm~\ref{algo:integer-relations}.
The heuristic check relies on assuming~$\mu(P) = \smash{\frac{m-\rk \Lambda}{p}}$,
approximating~$\phi(\beta)\simeq\beta^{1/\mu(P)}$, that is $\phi(\beta) \simeq {\beta^{\frac{p}{m-\rk \Lambda}}}$, and applying Proposition~\ref{prop:heuristic}.

\section{Developing an intrinsic measure of error}\label{sec:intrinsic_error}

Working with a finite approximation of periods, there is the possibility of miscomputing the Picard group. Although there will be no miscomputation if the periods are approximated to sufficient precision, we do not know \emph{a priori} what constitutes ``sufficient precision'' for any given example. We can not solve this problem here but we will attempt to facilitate an \emph{a posteriori} certification scheme.

More precisely, we define two intrinsic quantities~$B_{min}$ and~$\tau_N$ related to an algebraic surface~$X$. Then we give an algorithm that computes numerically the Picard group of~$X$ alongside with half-a-certificate $(B,N,\varepsilon) \in \Rr^3_{> 0}$ such that if~$B_{min} < B$ and~$\varepsilon < \tau_N$ then the computation of the Picard group is correct.
Unfortunately, the quantities~$B_{min}$ and~$\tau_N$ are not easily computable.

\subsection{Summary of the results in this section}

There is a canonical \emph{positive definite} inner product on $\H^2(X,\Rr)=\H^2(X,\Zz)\otimes_{\Zz}\Rr$ whose exact formulation we recall in Section~\ref{sec:canonical_norm}. Briefly, it is obtained from the intersection product on cohomology by flipping the sign of the intersection matrix on a certain subspace of $\H^2(X,\Rr)$. We will refer to the norm induced by this inner product as the \emph{canonical norm}.

Let us point out that the homology group $\H_2(X,\Rr)$ and the cohomology group $\H^2(X,\Rr)$ are canonically identified via Poincar\'e duality. If $\{\gamma_i\}_{i=1}^m$ is a basis for homology and the cohomology is equipped with the corresponding dual basis, then the identification $\H_2(X,\Rr) \isoto \H^2(X,\Rr)$ is given by the intersection product matrix $(\gamma_i \cdot \gamma_j)_{i,j}$. For the rest of this section, homology and cohomology will thus be identified. 

Let $\cp \in \Cc^{r\times m}$ be the period matrix of $X$ as in~Section~\ref{sec:periods_and_picard}. Unlike the norm, the matrix $\cp$ depends on the choice of a basis for homology and a basis for the space of holomorphic forms in $\H^2(X,\Cc)$. As outlined in Section~\ref{sec:periods_and_picard}, we can compute a finite approximation of $\cp$ and then, via lattice reduction, compute a lattice $\Lambda \subset \H^2(X,\Zz) \simeq \Zz^{m}$ meant to represent the Picard group of $X$.

We give an explicit construction in Section~\ref{sec:canonical_proofs} of a real number $B$ associated to the computation of $\Lambda$. Let $B_{min}$ be the infimum over real numbers $c$ such that $\pic(X)$ is generated by elements of canonical norm at most $c$. 

\begin{theorem}\label{thm:picard_in_lambda}
  If $B_{min} < B$ then $\pic(X) \subset \Lambda$.
\end{theorem}

Let $\cp_{\Rr} \in \Rr^{2r \times m}$ be the matrix obtained by joining the real and imaginary parts of $\cp$ vertically. Define $U \subset \H^2(X,\Rr)\simeq \Rr^m$ to be the subspace generated by the rows of $\cp_{\Rr}$. With respect to the Hodge decomposition, one would write $U = \H^2(X,\Rr) \cap \left( \H^{2,0}(X) \oplus \H^{0,2}(X) \right)$. Let $U^\perp$ be the orthogonal complement of $U$ with respect to the cap product on cohomology. 

In Section~\ref{sec:canonical_proofs} we give an explicit construction for a pair of positive real numbers $(N,\varepsilon)$ associated to our lattice $\Lambda$. Denote by $\tau_N$ the minimum non-zero distance of vectors of canonical norm at most $N$ in $\H^2(X,\Zz)$ to $U^{\perp}$, that is,
  \begin{equation}
    \tau_N = \min \{\operatorname{dist}_{U^{\perp}}(v) \mid v \in \H^2(X,\Zz) \setminus U^{\perp},\, \norm{v} \le N\}.
  \end{equation}

\begin{theorem}\label{thm:lambda_in_picard}
  If $\varepsilon < \tau_N$ then $\Lambda \subset \pic(X)$.
\end{theorem}

\noindent The two theorems stated above are proven in Section~\ref{sec:canonical_proofs}.

\subsection{Canonical norm}\label{sec:canonical_norm}

Let $X \subset \ppp$ be a smooth surface of  degree $d$. There is a canonical positive definite intersection product on $\H^2(X,\Rr)$~\cite[Example~3.1.7(ii)]{huybrechts--k3}. For the proofs of the statements made here, we refer to \emph{loc.\ cit.} %

If $(\gamma_{i}\cdot \gamma_j)_{i,j}$ is the intersection matrix on homology $\H_2(X,\Rr)$ with respect to a basis $\{\gamma_i\}_{i=1}^m$, then the inverse matrix $(\gamma_i \cdot \gamma_j)^{-1}_{i,j}$ represents the matrix for the cup product $(\gamma_i^* \cup \gamma_j^*)_{i,j}$ on cohomology $\H^2(X,\Rr)$ with respect to the dual basis $\{\gamma_i^*\}_{i=1}^m$. Note that the cup product on cohomology is available to us.

By abuse of notation, we will denote by $h_X$ the image of the polarization with respect to the identification $\H_2(X,\Rr) \isoto \H^2(X,\Rr)$. Recall that $U$ has been defined as the subspace of cohomology generated by the real and imaginary parts of a period matrix $\cp$ of $X$. Let  $U^\perp \subset \H^2(X,\Rr)$ be the space orthogonal to $U$. If we denote by $W \subset U^{\perp}$ the subspace in $U^{\perp}$ orthogonal to $h_X$ then we get an orthogonal decomposition of cohomology with respect to the cap product: 
\begin{equation}\label{eq:decomposition}
  \H^2(X,\Rr) = U \oplus W \oplus \Rr\langle h_X \rangle.
\end{equation}

The cup product on cohomology is positive definite on $U$ and $\Rr\langle h_X \rangle$ but negative definite on $W$. The canonical positive definite metric on $\H^2(X,\Rr)$ is obtained by flipping the sign of the cup product on $W$. In particular, for $v \in \H^2(X,\Rr)$ if we write $v=v_U + v_W + v_h$ by respecting the decomposition \eqref{eq:decomposition} then the canonical norm is given by:
\begin{equation}\label{eq:can_norm}
  \norm{v}^2 = v_U\cup v_U - v_W\cup v_W + v_h \cup v_h.
\end{equation}
It is easy to see that $v_h = \frac{v\cup h_X}{d} h_X$, where $d=\deg X = h_X \cup h_X$. Let us define $v_{prim}= v - v_h$. Substituting these two terms into \eqref{eq:can_norm} we get the following expression:
\begin{equation}\label{eq:can_norm_2}
  \norm{v}^2 = v\cup h_X - v_{prim}\cup v_{prim}+ 2v_{prim} \cup v_U .
  \end{equation}
 Let us emphasize that every term in this expression, with the exception of $v_U$, is available to us. We now turn to the problem of approximating $v_U$.
  
 \subsection{Rigorous bounds for the canonical norm}\label{sec:canonical_rigor}
  
  In order to use \eqref{eq:can_norm_2}, we want to find the orthogonal projection operator $P_U \colon \H^2(X,\Rr) \to U ; v\mapsto v_U$. This operator can be constructed from the period matrix $\cp$ of $X$. We note here that \eqref{eq:can_norm_2} varies continuously in $v_U$, that is, small errors in expressing $P_U$ will miscompute $\norm{v}$ by a small constant of proportionality.

  Let us denote by $\Ii \in \Zz^{m \times m}$ the matrix representing the cup product on cohomology, with respect to the trivialization $\H^2(X,\Zz)\simeq \Zz^m$ used to compute $\cp$. Recall that $\cp_{\Rr}$ is the vertical join of the real and imaginary parts of $\cp$. As the cup product is positive definite on $U$, we can define $\cp_{\Rr}^{on} \in \Rr^{2r\times m}$ to be the matrix whose rows are obtained by the Gramm--Schmidt orthonormalization process from the rows of $\cp_{\Rr}$ with respect to $\Ii$. In coordinates, the projection operator $P_U \colon \H^2(X,\Rr) \to U$ is given by the matrix $\cp_{\Rr}^{on} \cdot \Ii$.

  In practice, we only have a rational approximation $\cq_{\Rr}$ of $\cp_{\Rr}$ and a matrix $E$ such that the absolute value of the $(i,j)$-th entry of the difference $\cp_{\Rr}-\cq_{\Rr}$ is bounded from above by the $(i,j)$-th entry of $E$. In other words, we have a \emph{rectangular neighbourhood} containing $\cp_{\Rr}$. This allows us to compute an approximation $P_U^{\square}$ of the projection operator $P_U$ using Algorithm~\ref{alg:projector} such that $P_U^{\square}(v)$ returns a rectangular neighbourhood $v_U^{\square}$ of $v_U$. Substituting $v_U^{\square}$ for $v_U$ in \eqref{eq:can_norm_2} we get an interval which contains $\norm{v}$. 
  
  In particular, for Theorem~\ref{thm:lambda_in_picard} we need the ability to bound the distance to $U^{\perp}$. For any $v$, this distance is $\operatorname{dist}_{U^{\perp}}(v) = \norm{P_U (v)}$.  We can therefore bound $\operatorname{dist}_{U^{\perp}}(v)$ by computing $v_U^{\square}$ and taking the maximum of the interval $v_U^{\square} \cup v_U^{\square}$.

\begin{algorithm}[t] \centering
  \begin{description}[leftmargin=5em,style=nextline]
    \item[Input.] A rectangular neighbourhood of $\cp_{\Rr}$, that is, $(\cq_{\Rr},E)$.
    \item[Output.]  An operator $P^{\square}_U$ which takes $v \in \H^2(X,\Zz)$ and gives a rectangular neighbourhood containing $v_U$.
  \end{description} \medskip
  \begin{enumerate}
    \item Using ball arithmetic, orthonormalize the matrix of intervals with center $\cq_{\Rr}$ and radius $E$. Let $\cb$ denote the resulting matrix of intervals.
    \item The orthonormalization $\cp^{on}_\Rr$ of $\cp_{\Rr}$ is contained in $\cb$. Therefore, for any $v$ the operator $P^{\square}_U$ defined by $v \mapsto \cb \cdot \Ii \cdot v$ is what we want.
  \end{enumerate}
  \caption{Computing a rectangular neighbourhood of the projection operator.}
  \label{alg:projector}
\end{algorithm}

  \subsection{Computing half of the certificate}\label{sec:canonical_proofs}

  Let $\H^2(X,\Zz)$ be identified with $\Zz^{m}$ by the choice of a basis $\gamma_1^*,\dots,\gamma_m^*$ in which the period matrix $\cp$ is computed. For $v \in \H^2(X,\Rr)$ write $v=\sum_{i=1}^m v_i \gamma_i^*$ and define the \emph{coordinate norm} $|v| = \sqrt{\sum_{i=1}^m v_i^2}$. Recall that our approximation $\cq_\Rr$ of $\cp_\Rr$ comes with rigorous error bounds which we can collect into the matrix $E$ to form the rectangular neighbourhood $(\cq_\Rr,E)$.
  
  Let $\Lambda \subset \Zz^m$ be the candidate Picard group of $X$ computed using the approximation $\cq_\Rr$ of $\cp_{\Rr}$ as in Section~\ref{sec:periods_and_picard}. The lattice $\Lambda$ is computed according to an exact procedure detailed in Section~\ref{sec:numer-reconstr-picar}. In order to claim that $\Lambda$ contains $\pic(X)$ we need a careful study the lattice reduction algorithm by which $\Lambda$ is produced. This study takes place in Section~\ref{sec:numer-reconstr-picar} and that section allows us to compute $\Lambda$ together with a real number $B_\Lambda$. This number $B_\Lambda$ has the property that a vector $v \in \pic(X) \subset \Zz^m$ is contained in $\Lambda$ if the coordinate norm $|v|$ is less than $B_\Lambda$. We will now unshackle this bound from the coordinates chosen and phrase it in terms of the canonical norm.

  \begin{lemma}\label{lem:canonical_vs_coordinate}
  For any $v \in \H^2(X,\Zz)$ the canonical norm is bounded by the coordinate norm as follows: $|v| \left( \sum_{i=1}^m \norm{\gamma_i^*}^{-2} \right)^{-1/2} \le \norm{v} \le |v| \left( \sum_{i=1}^m \norm{\gamma_i^*}^2 \right)^{1/2}$.
\end{lemma}
\begin{proof}
  Use the triangle inequality and Cauchy--Schwartz on $\norm{\cdot}$ for the upper bound. The lower bound follows from the observation $|v_i| \norm{\gamma_i^*} \le \norm{v}$, as $\norm{\cdot}$ is induced from an inner product.
\end{proof}

In particular, upper and lower bounds of the canonical norms of $\gamma_i^*$ can be computed as in Section~\ref{sec:canonical_rigor}. Using these bounds and Lemma~\ref{lem:canonical_vs_coordinate} we can compute $\xi_1,\xi_2 \in \Rr_{>0}$ such that for every $v \in \H^2(X,\Rr)$ we have $\xi_1 |v| \le \norm{v} \le \xi_2 |v|$. We now define $B$ to be $\xi_2 B_\Lambda$.

\begin{proof}[Proof of Theorem~\ref{thm:picard_in_lambda}]
   By construction of $B_\Lambda$, if $v \in \pic(X)$ satisfies $|v| < B_\Lambda$ then $v \in \Lambda$. Using Lemma~\ref{lem:canonical_vs_coordinate} and by definition of $B$ we have $\norm{v} \le \xi_2|v| < \xi_2 B_\lambda=B$. By hypothesis, $\pic(X)$ is generated by its elements of canonical norm at most $B_{min} < B$. Therefore, a generating set for $\pic(X)$ is contained in $\Lambda$.
\end{proof}

Let $v_1,\dots,v_{\rho} \in \Lambda$ be a basis. Using Section~\ref{sec:canonical_rigor} compute for each $i$ the interval $\norm{v_i}^{\square}$ containing the canonical norm $\norm{v_i}$ and compute the interval containing the distance $\dist_{U^\perp}(v_i)$. Using these intervals we define $N,\varepsilon \in \Rr_{>0}$ such that $\norm{v_i} <N$ and $\dist_{U^\perp}(v_i) < \varepsilon$. 

\begin{proof}[Proof of Theorem~\ref{thm:lambda_in_picard}]
  An element $v\in \H^2(X,\Zz)$ is in $\pic(X)$ if and only if $\dist_{U^\perp}=0$ by Lefschetz (1,1) theorem. If any $v$ of $\H^2(X,\Zz)$ having canonical norm at most $N$ is either in $\pic(X)$ or satisfies $\dist_{U^\perp}(v) > \varepsilon$ then, by definition of $N$ and $\epsilon$, $\Lambda \subset \pic(X)$. The hypothesis $\varepsilon < \tau_N$ ensures precisely this condition. 
\end{proof}
  
\section{Hypersurfaces of arbitrary even dimension}\label{sec:higher_degree_dimension}

Let $k$ be a positive integer and let $X \subset \Pp^{2k+1}$ be a smooth hypersurface. Using Lefschetz hyperplane theorem and Poincar\'e duality we see that the cohomology groups $\H^i(X,\Zz)$ are either trivial or $\Zz$ except when $i=2k$. The Hodge decomposition on de Rham cohomology gives
\begin{equation}
  \H_{\dR}^{2k}(X,\Cc) = \bigoplus_{p+q=2k} \H^{p,q}(X,\Cc).
\end{equation}
Algebraic cycles of dimension $k$ in $X$ give cohomology classes in 
\begin{equation}
  \hodge^k(X) \eqdef \H^{k,k}(X,\Cc)\cap \H^{2k}(X,\Zz). 
\end{equation}
As a generalization of Theorem~\ref{thm:lefschetz11},
Hodge conjecture predicts that the vector space $\hodge^k(X)\otimes_{\Zz}\Qq$ is
spanned by algebraic cycles~\cite{voisin--hodge_conjecture}. %

The Hodge group $\hodge^k(X)$ comes with an intersection pairing obtained by restricting the cup product on cohomology $\H^{2k}(X,\Cc)$. Furthermore, there is a \emph{polarization} $h_X^{k} \in \hodge^k(X)$ where $h_X$ is the class of a generic hyperplane section of $X$.  The tools we used to tackle the computation of Picard groups apply to the following following problem:

{\em
  Given the defining equation of $X\subset\Pp^{2k+1}$, compute the rank $\rho$ of $\hodge^k(X)$, the $\rho\times\rho$ matrix of the intersection product and the $\rho$ coordinates of the polarization $h_X^{k}$ in some basis of~$\hodge^k(X) \simeq \Zz^\rho$.
\/}

Suppose now that $\gamma_1,\dots,\gamma_m \in \H_{2k}(X,\Zz)$ is a basis for the middle homology group of $X$. We can then identify the cohomology $\H^{2k}(X,\Cc) = \operatorname{Hom}(\H_{2k}(X,\Zz),\Cc)$ with $\Cc^m$ via the dual basis of $\{\gamma_i\}_{i=1}^{m}$. Let us write $F^{2k,\ell}(X,\Cc) = \bigoplus_{j=0}^{\ell} \H^{2k-j,j}(X,\Cc)$ for the corresponding Hodge filtration.

Let $\omega_1,\dots,\omega_s \in F^{2k,k-1}(X,\Cc)$ be a basis for the $(k-1)$-th part of the Hodge filtration. Suppose that we have the coordinates of $\omega_i$ with respect to the identification $\H^n(X,\Cc)\simeq \Cc^m$, that is, suppose that for each $i=1,\dots,s$ the following integrals are known
\begin{equation}
  \left(\int_{\gamma_1} \omega_i,\dots,\int_{\gamma_m} \omega_i\right) \in \Cc^m.
\end{equation}
These \emph{periods} of $X$ are listed as the columns of the following matrix:
\begin{equation}\label{eq:hdg_P}
  \cp \eqdef \left( \int_{\gamma_i} \omega_j \right)_{\substack{i=1,\dots,m \\ j=1,\dots,s}}.
\end{equation}
The matrix $\cp$ induces the linear map $\cp_\Zz\colon\Zz^{m} \to \Cc^s$ by acting on the integral vectors from the right.

\begin{lemma}\label{lem:kernel}
  We have a natural isomorphism $\hodge^k(X) \simeq \ker \cp_\Zz$.
\end{lemma}
\begin{proof}
  The kernel of $\cp_{\Zz}$ computes in $\H_{2k}(X,\Zz)$ the classes annihilated by $F^{2k,k-1}(X,\Cc)$. Any integral (or real) class annihilated by $F^{2k,k-1}$ will also be annihilated by its complex conjugate $\overline{F^{2k,k-1}}$. We now use the equality $\H^{k,k}(X,\Cc) = \left( F^{2k,k-1} \oplus \overline{F^{2k,k-1}} \right)^\perp $ and the definition of Poincar\'e duality. 
\end{proof}

  The kernel of $\cp_{\Zz}$ sits most naturally in homology $\H_{2k}(X,\Zz)$ and is denoted by $\hodge_k(X)$. Using Sert\"oz~\cite{sertoz18} we can approximate the matrix $\cp$ to the desired degree of accuracy for an automatically generated basis of $F^{2k,k-1}(X,\Cc)$ and some implicit basis of $\H_{2k}(X,\Zz)$. The basis of $\H_{2k}(X,\Zz)$ comes with an intersection pairing as well as the coordinates of the polarization $h_{X}^k$ in this basis. 

In light of Lemma~\ref{lem:kernel}, we need to compute integral linear relations between the columns of $\cp$,
that we compute numerically and we recover~$\hodge_k(X)$ with lattice reduction algorithms, see~\S \ref{sec:numer-reconstr-picar}.
To be more precise, identify $\H_{2k}(X,\Zz)$ with $\Zz^m$ by choosing a basis. In turn, $\hodge_k(X)$ is identified with a sublattice $\Lambda \subset \Zz^m$. What we can compute is the following sublattice for some $B,\, \varepsilon \in \Rr_{> 0}$:
\begin{equation}
  \Lambda_{B,\varepsilon} = \{ a \in \Zz^m \mid \norm{\cp_{\Zz}(a)} < \varepsilon,\, \norm{a} < B\}.
\end{equation}

\begin{example}
  With minimal effort, the period computations explained in \cite{sertoz18} can be extended to mildly singular hypersurfaces in $\Pp^{2k+1}$ for $k>1$. We computed the periods of the Delsarte surface cut out by the quintic polynomial
  \begin{equation}
    x^5 + xzw^3 + y^5 + z^4w.
  \end{equation}
  We see that the Picard number of this surface is 13, in agreement with \cite[\S6]{schuett15}.
\end{example}

The study of the Hodge groups of cubic fourfolds is an active area of research~\cite{ranestad17, addington17}. Although generic cubic fourfolds provide a computational challenge, we can quickly compute the Hodge rank of sparse cubic fourfolds if most of the monomial terms are cubes of a single variable. 

\begin{example}
 Let $X$ be the cubic fourfold in $\Pp^5$ cut out by the equation 
 \begin{equation}
   6x_0^3 + 10x_0x_2x_4 + 9x_0x_2x_5 + 4x_1^3 + 2x_1x_2^2 + 4x_2^3 + 3x_3^3 + 4x_4^3 + 9x_5^3.
 \end{equation}
 We find that $\hodge^2(X)$ is of rank 13. We used 400 digits of precision and can certifiably say that we computed $\Lambda_{B,\varepsilon}$ with $B \approx 10^{70}$ and $\varepsilon \approx 10^{-330}$, see Lemma~\ref{lem:project_reduction}.
\end{example}

\section{Experimental results}
\label{sec:experiment}

\begin{figure}[t]
  \centering
  \begin{tabular}{cc}
  \toprule
  Defining polynomial & Picard number \\
  \midrule
    $-26x^4 - 88x^3y + 32y^3z + 93z^3w + 46w^4$ & 1\\
  $x^3y + z^4 + y^3w + zw^3$ & 4 \\
  $x^3y + y^4 + z^3w + yw^3 + zw^3$ & 6 \\
  $x^3y + z^4 + y^3w + xw^3 + w^4$ & 8 \\
  $x^3y + z^4 + y^3w + w^4$ & 10 \\
  $x^3y + y^4 + z^3w + x^2w^2 + w^4$ & 12 \\
  $x^3y + y^4 + z^3w + yw^3 + w^4$ & 14 \\
  $x^3y + y^3z + z^4 + xy^2w + zw^3$ & 15 \\
  $x^3y + y^4 + z^3w + xyw^2 + y^2w^2 + w^4$ & 16 \\
  $x^3y + y^4 + z^4 + x^2w^2 + zw^3$ & 17 \\
  $x^3y + x^3z + y^3z + yz^3 + w^4$ & 18 \\
  $x^3y + z^4 + y^3w + xyzw + xw^3$ & 19 \\
  $x^3y + z^4 + y^3w + xw^3$ & 20 \\
  \bottomrule
\end{tabular}
\caption{Specimen polynomial for each Picard number found}
\label{fig:picspecimen}
\end{figure}

We documented the Picard groups of $2790$ quartic surfaces defined by sparse polynomials.
For these computations, setting up the initial value problem for the periods was not the limiting factor but rather the numerical solution of these initial value problems took the greatest amount of time.
With our current methods, the periods of a quartic surface defined by dense polynomials can not be computed in a reasonable amount of time, say, less than a day.

The database presented here relies on a systematic exploration of quartics that are defined by a sum of at most six monomials in $x,y,z,w$ with coefficients $0$ or $1$. We built a graph whose vertices store the defining polynomials and an edge between two polynomial is constructed if the difference of the two polynomials is supported on at most two monomials (this is done to ease the computations).
Then, for each edge, we setup and attempt to solve the initial value problem defining the transition matrix from one set of periods onto the other, using 300 decimal digits of precision, see~\cite{sertoz18} for details.
Computation is stopped if it takes longer than an hour and the edge deleted.
Having explicit formulas for the periods of Fermat surface $\left\{ x^4+y^4+z^4+w^4 = 0 \right\}$,
we can compute the periods of any vertex in the connected component of $x^4+y^4+z^4+w^4$ in the resulting graph by simply multiplying the transition matrices of each edge along a path.

For each of the $2790$ polynomials in our database, we computed the Picard group, the polarization, the intersection product, the endomorphism ring and the number of smooth rational curves of degree up to 4.\footnote{Results are compiled at \url{https://pierre.lairez.fr/picdb}. This page will be continuously updated.} We found quartic surfaces with Picard number $4, 6, 8, 10, 12$, $14,15,16,17,18,19,20$, see Figures~\ref{fig:picspecimen} and~\ref{fig:picrk}.
When possible, we checked that our results were consistent with Shioda's formula
for $4$-nomial quartic surfaces~\cite{shioda--delsarte}, reduction methods with Costa's
implementation~\cite{costa-phd} and symbolic line counting.

\begin{example}
  $\{ x^4 + y^3z + xyzw + z^3w + yw^3 = 0 \}$.
  This surface has Picard number 19. It contains no smooth rational curves of degree $<4$, and 133056~smooth rational curves of degree~4. They generate the Picard group.
\end{example}

\begin{example}
  $\{ x^3y + x^3z + y^3z + yz^3 + z^4 + xw^3 = 0 \}$.
  This surface has Picard number 10. It contains 13 lines that generate the Picard group.
  The endomorphism ring is~$\Qq(\exp(\frac{2\pi i}{18}))$, a cyclotomic extension
  of~$\Qq$ of degree~$6$.
  Up to degree~10, the smooth rational curves inside~$X$ count as follows.
  
  \medskip\centering
  \begin{tabular}{lcccccccccc}\toprule
    $d$ & 1 & 2 & 3 & 4 & 5 & 6 & 7 & 8& 9 & 10 \\
    $\# \mathcal{R}_d$ & 13 & 0 & 0 &  108& 0& 0& 972& 0& 0& 3996\\
    \bottomrule
  \end{tabular}
  \medskip
\end{example}

\begin{example}
  $\{ -26x^4 - 88x^3y + 32y^3z + 93z^3w + 46w^4 = 0 \}$.
  This surface has Picard number 1. (It is not part of the main dataset, we searched specifically for such a surface.)
\end{example}

\begin{example}
  $\{ x^3y + z^4 + y^3w + zw^3 = 0\}$.
  This surface has Picard number $4$. It contains exactly 4 lines that generate
  the Picard group and no other smooth rational curve of degree~$<100$.
  The endomorphism ring is~$\Qq(\exp(\frac{2\pi i}{54}))$, a cyclotomic extension
  of~$\Qq$ of degree~$18$.
\end{example}

\begin{figure}[t]\centering
\begingroup
  \makeatletter
  \providecommand\color[2][]{%
    \GenericError{(gnuplot) \space\space\space\@spaces}{%
      Package color not loaded in conjunction with
      terminal option `colourtext'%
    }{See the gnuplot documentation for explanation.%
    }{Either use 'blacktext' in gnuplot or load the package
      color.sty in LaTeX.}%
    \renewcommand\color[2][]{}%
  }%
  \providecommand\includegraphics[2][]{%
    \GenericError{(gnuplot) \space\space\space\@spaces}{%
      Package graphicx or graphics not loaded%
    }{See the gnuplot documentation for explanation.%
    }{The gnuplot epslatex terminal needs graphicx.sty or graphics.sty.}%
    \renewcommand\includegraphics[2][]{}%
  }%
  \providecommand\rotatebox[2]{#2}%
  \@ifundefined{ifGPcolor}{%
    \newif\ifGPcolor
    \GPcolortrue
  }{}%
  \@ifundefined{ifGPblacktext}{%
    \newif\ifGPblacktext
    \GPblacktexttrue
  }{}%
  \let\gplgaddtomacro\g@addto@macro
  \gdef\gplbacktext{}%
  \gdef\gplfronttext{}%
  \makeatother
  \ifGPblacktext
    \def\colorrgb#1{}%
    \def\colorgray#1{}%
  \else
    \ifGPcolor
      \def\colorrgb#1{\color[rgb]{#1}}%
      \def\colorgray#1{\color[gray]{#1}}%
      \expandafter\def\csname LTw\endcsname{\color{white}}%
      \expandafter\def\csname LTb\endcsname{\color{black}}%
      \expandafter\def\csname LTa\endcsname{\color{black}}%
      \expandafter\def\csname LT0\endcsname{\color[rgb]{1,0,0}}%
      \expandafter\def\csname LT1\endcsname{\color[rgb]{0,1,0}}%
      \expandafter\def\csname LT2\endcsname{\color[rgb]{0,0,1}}%
      \expandafter\def\csname LT3\endcsname{\color[rgb]{1,0,1}}%
      \expandafter\def\csname LT4\endcsname{\color[rgb]{0,1,1}}%
      \expandafter\def\csname LT5\endcsname{\color[rgb]{1,1,0}}%
      \expandafter\def\csname LT6\endcsname{\color[rgb]{0,0,0}}%
      \expandafter\def\csname LT7\endcsname{\color[rgb]{1,0.3,0}}%
      \expandafter\def\csname LT8\endcsname{\color[rgb]{0.5,0.5,0.5}}%
    \else
      \def\colorrgb#1{\color{black}}%
      \def\colorgray#1{\color[gray]{#1}}%
      \expandafter\def\csname LTw\endcsname{\color{white}}%
      \expandafter\def\csname LTb\endcsname{\color{black}}%
      \expandafter\def\csname LTa\endcsname{\color{black}}%
      \expandafter\def\csname LT0\endcsname{\color{black}}%
      \expandafter\def\csname LT1\endcsname{\color{black}}%
      \expandafter\def\csname LT2\endcsname{\color{black}}%
      \expandafter\def\csname LT3\endcsname{\color{black}}%
      \expandafter\def\csname LT4\endcsname{\color{black}}%
      \expandafter\def\csname LT5\endcsname{\color{black}}%
      \expandafter\def\csname LT6\endcsname{\color{black}}%
      \expandafter\def\csname LT7\endcsname{\color{black}}%
      \expandafter\def\csname LT8\endcsname{\color{black}}%
    \fi
  \fi
    \setlength{\unitlength}{0.0500bp}%
    \ifx\gptboxheight\undefined%
      \newlength{\gptboxheight}%
      \newlength{\gptboxwidth}%
      \newsavebox{\gptboxtext}%
    \fi%
    \setlength{\fboxrule}{0.5pt}%
    \setlength{\fboxsep}{1pt}%
\begin{picture}(7200.00,2592.00)%
    \gplgaddtomacro\gplbacktext{%
      \csname LTb\endcsname%
      \put(421,186){\makebox(0,0){\strut{}$1$}}%
      \csname LTb\endcsname%
      \put(753,186){\makebox(0,0){\strut{}$2$}}%
      \csname LTb\endcsname%
      \put(1085,186){\makebox(0,0){\strut{}$3$}}%
      \csname LTb\endcsname%
      \put(1417,186){\makebox(0,0){\strut{}$4$}}%
      \csname LTb\endcsname%
      \put(1749,186){\makebox(0,0){\strut{}$5$}}%
      \csname LTb\endcsname%
      \put(2080,186){\makebox(0,0){\strut{}$6$}}%
      \csname LTb\endcsname%
      \put(2412,186){\makebox(0,0){\strut{}$7$}}%
      \csname LTb\endcsname%
      \put(2744,186){\makebox(0,0){\strut{}$8$}}%
      \csname LTb\endcsname%
      \put(3076,186){\makebox(0,0){\strut{}$9$}}%
      \csname LTb\endcsname%
      \put(3408,186){\makebox(0,0){\strut{}$10$}}%
      \csname LTb\endcsname%
      \put(3740,186){\makebox(0,0){\strut{}$11$}}%
      \csname LTb\endcsname%
      \put(4072,186){\makebox(0,0){\strut{}$12$}}%
      \csname LTb\endcsname%
      \put(4404,186){\makebox(0,0){\strut{}$13$}}%
      \csname LTb\endcsname%
      \put(4736,186){\makebox(0,0){\strut{}$14$}}%
      \csname LTb\endcsname%
      \put(5068,186){\makebox(0,0){\strut{}$15$}}%
      \csname LTb\endcsname%
      \put(5399,186){\makebox(0,0){\strut{}$16$}}%
      \csname LTb\endcsname%
      \put(5731,186){\makebox(0,0){\strut{}$17$}}%
      \csname LTb\endcsname%
      \put(6063,186){\makebox(0,0){\strut{}$18$}}%
      \csname LTb\endcsname%
      \put(6395,186){\makebox(0,0){\strut{}$19$}}%
      \csname LTb\endcsname%
      \put(6727,186){\makebox(0,0){\strut{}$20$}}%
    }%
    \gplgaddtomacro\gplfronttext{%
      \csname LTb\endcsname%
      \put(421,558){\makebox(0,0){\strut{}0}}%
      \csname LTb\endcsname%
      \put(753,558){\makebox(0,0){\strut{}0}}%
      \csname LTb\endcsname%
      \put(1085,558){\makebox(0,0){\strut{}0}}%
      \csname LTb\endcsname%
      \put(1417,1166){\makebox(0,0){\strut{}249}}%
      \csname LTb\endcsname%
      \put(1749,558){\makebox(0,0){\strut{}0}}%
      \csname LTb\endcsname%
      \put(2080,858){\makebox(0,0){\strut{}123}}%
      \csname LTb\endcsname%
      \put(2412,558){\makebox(0,0){\strut{}0}}%
      \csname LTb\endcsname%
      \put(2744,1012){\makebox(0,0){\strut{}186}}%
      \csname LTb\endcsname%
      \put(3076,558){\makebox(0,0){\strut{}0}}%
      \csname LTb\endcsname%
      \put(3408,2405){\makebox(0,0){\strut{}757}}%
      \csname LTb\endcsname%
      \put(3740,558){\makebox(0,0){\strut{}0}}%
      \csname LTb\endcsname%
      \put(4072,1507){\makebox(0,0){\strut{}389}}%
      \csname LTb\endcsname%
      \put(4404,558){\makebox(0,0){\strut{}0}}%
      \csname LTb\endcsname%
      \put(4736,1610){\makebox(0,0){\strut{}431}}%
      \csname LTb\endcsname%
      \put(5068,619){\makebox(0,0){\strut{}25}}%
      \csname LTb\endcsname%
      \put(5399,1141){\makebox(0,0){\strut{}239}}%
      \csname LTb\endcsname%
      \put(5731,629){\makebox(0,0){\strut{}29}}%
      \csname LTb\endcsname%
      \put(6063,1251){\makebox(0,0){\strut{}284}}%
      \csname LTb\endcsname%
      \put(6395,717){\makebox(0,0){\strut{}65}}%
      \csname LTb\endcsname%
      \put(6727,590){\makebox(0,0){\strut{}13}}%
    }%
    \gplbacktext
    \put(0,0){\includegraphics{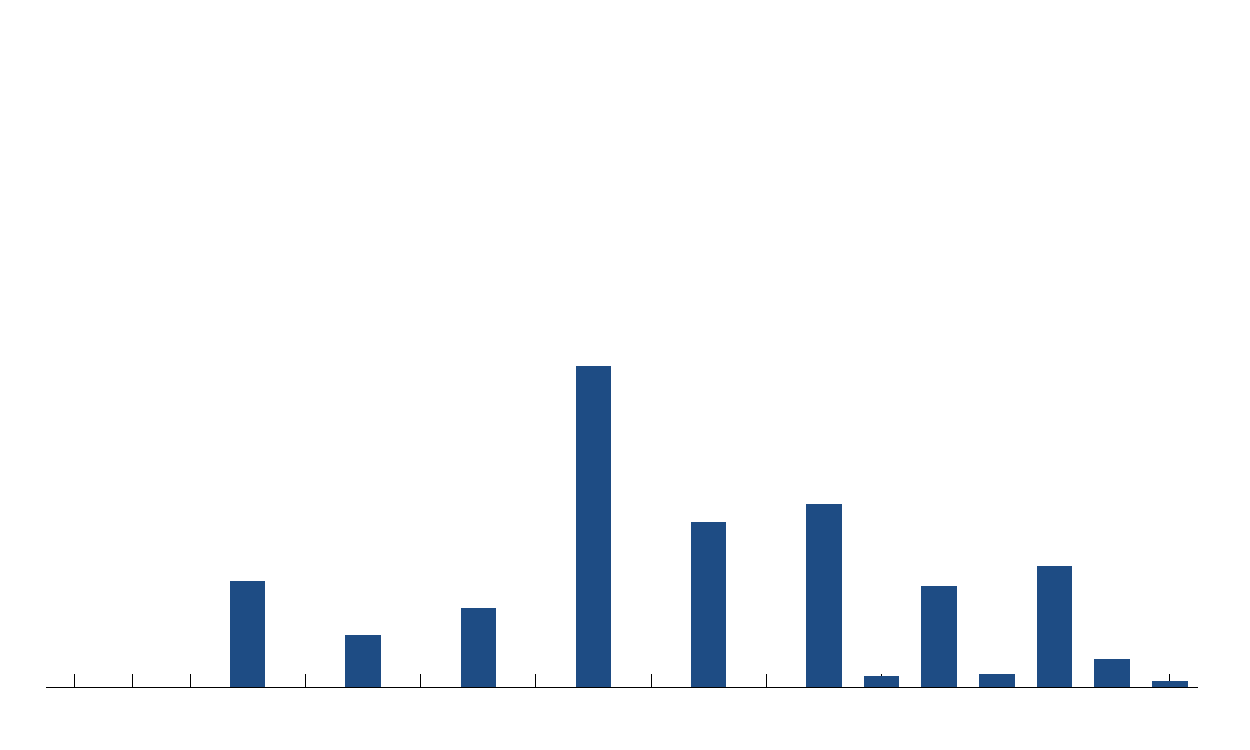}}%
    \gplfronttext
  \end{picture}%
\endgroup
 
  \caption{Occurence of each possible values of the Picard number in the dataset}
\label{fig:picrk}
\end{figure}

\section{Computing the coordinates of the polarization}\label{sec:completing_homology}

Let $X = Z(f_X) \subset \Pp^{n+1}$ be a smooth hypersurface of degree $d$ and assume $n$ is even. We compute a basis for the middle integral homology $\H_n(X,\Zz)$ by carrying over a basis from a hypersurface of Fermat type \cite[\S1.3]{sertoz18}. If $h_X = [X \cap H]$ denotes the hyperplane class in $X$, then $h_X^{n/2} \in \H_n(X,\Zz)$ is the polarization. The orthogonal complement of $h_X^{n/2}$ is the \emph{primitive homology}, denoted $P\H_{n}(X,\Zz)$. In order to compute the periods of $X$, it is sufficient to work only with the primitive homology $P\H_n(X,\Zz)$ as is done in \cite{sertoz18}. In \S4.5 of \emph{loc.\ cit.}\ there is a sketch on how to complete the given basis for the primitive homology to a basis of homology. In this section we flesh out the details as the particular choices we make in completing the basis determine the coordinates of the polarization. The problem that must be addressed is that $h_X^{n/2}$ and $P\H_n(X,\Zz)$ do not generate $\H_n(X,\Zz)$ but a full rank sublattice. 

In \cite{sertoz18} the Fermat surface $Y=Z(x_0^d+\dots+x_{n}^d - x_{n+1}^d)$ was used for the construction of a basis of primitive homology. This basis is formed using a \emph{Pham cycle} and the Pham cycle itself is formed by gluing translates of the following simplex:
\[
  D = \{[s_0: s_1 : \dots : s_n : 1] \mid s_i \in [0,1],\, s_0^d + \dots + s_n^d = 1\} \subset Y.
\]
For $\beta=(\beta_0,\dots,\beta_{n+1}) \in \Zz^{n+2}$ we define the translations $t^\beta \colon \Pp^{n+1} \to \Pp^{n+1}$ by the action on the coordinates $x_i \mapsto \xi^{\beta_i}x^i$. Then the Pham cycle $S$ is defined by:
\[
  S = (1-t_0^{-1})\cdots (1-t_n^{-1}) D,
\]
where summation is union and negation is change of orientation~\cite{pham--fermat}. It is possible to compute a subset $B \subset \Zz^{n+2}$ for which the set $\{t^\beta S\}_{\beta \in B}$ is a basis for the primitive homology $P\H_n(Y,\Zz)$, for instance, use Corollary~4.8~\cite{sertoz18}. We will now add one more cycle to complete $\{t^{\beta} S\}_{\beta \in B}$ to a basis of homology.

With $d$ being the degree of $X$ and $Y$, we denote the $d$-th root of $-1$ by $\eta := \exp(\frac{\pi\sqrt{-1}}{d})$ and the $d$-th root of $1$ by $\xi := \exp(\frac{2\pi\sqrt{-1}}{d})$. 
Let $\Pp^{n/2}$ be a projective space with coordinate functions $\mu_0,\dots,\mu_{n/2}$ and consider the linear map $\Pp^{n/2} \to \Pp^{n+1}$ defined by
\begin{align}
  x_{2k} = \mu_k,\,\, x_{2k+1} = \eta \mu_k & \quad k=0,\dots,\frac{n}{2}-1, \\
  x_n = \mu_{\frac{n}{2}} ,\,\, x_{n+1} = \mu_{\frac{n}{2}}. & \nonumber
\end{align}
The image of this map is a linear space $L$ which is evidently contained in $Y$. Let $[L]$ be the homology class of $L$ and let $\gamma_{\beta}$ be the homology class of $t^{\beta} S$. The set $\{[L]\} \cup \{\gamma_\beta\}_{\beta \in B}$ is a basis for the integral homology $\H_n(X,\Zz)$.

As $Y$ is deformed into $X$, the homology class of $L$ will typically deform into a class which no longer supports an algebraic subvariety and therefore this class will typically have non-zero periods. Nevertheless, we can deduce the periods of $L$ as it deforms based on the following two observations: The polarization $h_Y^{n/2}$ deforms in to $h_X^{n/2}$ and will always remain algebraic throughout the deformation. We will know the periods of the Pham basis $\{t^\beta S\}_{\beta \in B}$ as it deforms. 

The homology with rational coefficients $\H_n(Y,\Qq)$ splits into the direct sum $P\H_n(Y,\Qq)\oplus \Qq\langle h_Y^{n/2} \rangle$ so that we may write:
\begin{equation}\label{eq:LinB}
  [L] = \frac{1}{d} h_{Y}^{\frac{n}{2}} + \sum_{\beta \in B} a_{\beta} \gamma_\beta.
\end{equation}
The coefficients $\{a_{\beta}\}_{\beta \in B} \subset \Qq$ of this relation remain constant as we carry the basis $\{[L]\} \cup \{\gamma_\beta\}_{\beta \in B}$ to a basis of $\H_n(X,\Zz)$. The problem of computing the periods of $L$ as it deforms is therefore reduced to computing the coefficients $\{a_\beta\}_{\beta \in B}$.  Put an ordering on $B$ and let 
\begin{equation}\label{eq:abl}
  a_{B,L} = (a_{\beta})_{\beta \in B} \in \Qq^{\#B} 
\end{equation}
denote the row vector of coefficients defined in (\ref{eq:LinB}).

  Let $b_{B,L} = ([L] \cdot \gamma_\beta)_{\beta \in B} \in \Qq^{\#B}$ be the intersection numbers of $L$ with the Pham basis and let $M_{B} = (\gamma_{\beta} \cdot \gamma_{\beta'})_{\beta,\beta' \in B}$ be the matrix of intersections of the Pham basis. We see that $a_{B,L} = M_B^{-1} b_{B,L}$ so it remains to compute $M_{B}$ and $b_{B,L}$.

Fix $d \ge 2$ and define a function $\chi\colon \Zz \to \{-1,0,1\}$ as follows:
\[
  \chi(b) = \begin{cases}
    1 & \text{if } b = 0 \mod d\\
    -1 & \text{if } b \equiv 1 \mod d  \\
    0 & \text{if } b \not\equiv 0,1 \mod d.
  \end{cases} 
\]

\begin{proposition}\label{prop:intersection}
  For $\beta=(\beta_i)_{i=0}^{n+1}, \beta'=(\beta'_i)_{i=0}^{n+1} \in \Zz^{n+2}$ let $\beta'' = (\beta_i-\beta_i' - \beta_{n+1}+\beta_{n+1}')_{i=0}^{n+1}$. The Pham cycles $t^\beta S$ and $t^{\beta'} S$ intersect as follows:
  \[
    \langle t^{\beta} S , t^{\beta'} S \rangle = (-1)^{\frac{(n+1)n}{2}} \left( \prod_{b \in \beta''} \chi(b) - \prod_{b \in \beta''} \chi(b+1) \right).
  \]
\end{proposition}

For a proof of Proposition~\ref{prop:intersection} see any one of \cite{arnold-1984, movasati--hodge, looijenga-10}. We reformulated the statement here for the choices that were made in \cite{sertoz18} and in the style that was first communicated to us by Degtyarev and Shimada.

  Define the function $\tau_d: \Zz \to \{-1,0,1\}$ where:
  \[
    \tau_d(i) = \left\{ \begin{array}[]{cl}
      1 & i \equiv 1 \imod{2d} \\
      -1 & i \equiv -1 \imod{2d} \\
      0 & \text{otherwise.}
    \end{array}\right.
  \]

\begin{lemma}\label{lem:line}
  The intersection pairing of the linear space $L$ with the translates of the Pham cycle $S$ can be expressed as follows:
  \[
    \langle L,t^\beta S \rangle =\tau_d(2\beta_n - 2\beta_{n+1}-1) \prod_{i=0}^{\frac{n}{2}-1} \tau_d(2\beta_{2i} - 2\beta_{2i+1} + 1).
  \]
\end{lemma}

\noindent Lemma~\ref{lem:line} is proven by a straightforward application of Theorem 2.2 in \cite{degtyarev-16}.

\begin{example}
Let us consider quartic surfaces in $\ppp$, that is $d=4$ and $n=2$. Using Corollary~4.8~\cite{sertoz18} we find
\begin{multline*}
  B=\{ 
      ( 0, 0, 0, 0 ), ( 0, 0, 1, 0 ), ( 0, 1, 0, 0 ), ( 1, 0, 0, 0 ), ( 0, 0, 2, 0 ), ( 0, 1, 1, 0 ), ( 1, 0, 1, 0 ), \\
    ( 0, 2, 0, 0 ), ( 1, 1, 0, 0 ), ( 2, 0, 0, 0 ), ( 0, 1, 2, 0 ), ( 1, 0, 2, 0 ), ( 0, 2, 1, 0 ), ( 1, 1, 1, 0 ), \\
  ( 2, 0, 1, 0 ), ( 1, 2, 0, 0 ), ( 2, 1, 0, 0 ), ( 0, 2, 2, 0 ), ( 1, 1, 2, 0 ), ( 2, 0, 2, 0 ), ( 1, 2, 1, 0 ) 
\}.
\end{multline*}
With respect to this basis, and the ordering presented above, we find that the vector $a_{B,L}$ of \eqref{eq:abl} is given by
\begin{equation}
  a_{B,L} = (0, -1, \tfrac12, 0, 0, \tfrac12, -1, 0, \tfrac12, 0, \tfrac34, \tfrac14, -\tfrac14, -\tfrac12, -\tfrac14, -\tfrac34, \tfrac14, 0, 0, \tfrac12, -\tfrac12).
\end{equation}
The set $\{\gamma_\beta\}_{\beta\in B}$ is completed to a basis with the addition of $[L]$. In this basis, Equation~\eqref{eq:LinB} gives us the coordinates of the polarization:
\begin{equation}
  h_X = (0, 4, -2, 0, 0, -2, 4, 0, -2, 0, -3, -1, 1, 2, 1, 3, -1, 0, 0, -2, 2, 4).
\end{equation}
\end{example}

\bibliographystyle{siamplain}
\bibliography{my,maths}

\end{document}